\documentclass[12pt]{article}
\usepackage{amsthm,amsfonts,amssymb,amsmath}
\usepackage{mathtools, comment}
\usepackage{cite}
\usepackage{epsfig}
\usepackage{url}
\usepackage{xcolor,tikz}
\usetikzlibrary{matrix}
\usepackage{multicol}
\usepackage{authblk,fullpage,breqn,hyperref}
\usepackage{stmaryrd}
\usepackage[shortlabels]{enumitem}

\numberwithin{equation}{section}

\setlist{leftmargin=3\parindent,labelindent=3\parindent}
\setlist[enumerate]{%
  leftmargin=3\parindent,%
  align=left,%
  labelwidth=3\parindent,%
  labelsep=0pt%
}
\setlist[enumerate,1]{%
  label={\normalfont (\thesection.\arabic{equation})}, ref={\normalfont \thesection.\arabic{equation}},
  resume%
}

\newtheorem{thm}[equation]{Theorem}
\newtheorem{cor}[equation]{Corollary}
\newtheorem{lem}[equation]{Lemma}

\newtheorem{conj}[equation]{Conjecture}

\theoremstyle{definition}
\newtheorem{defn}[equation]{Definition}

\newtheorem{rem}[equation]{Remark}
\newtheorem{obs}[equation]{Observation}

\newtheorem*{ack}{Acknowledgements}

\newtheorem{ques}[equation]{Question}
\newtheorem{prob}[equation]{Problem}

\theoremstyle{remark}

\newcommand{\newConj}[2]{
  \expandafter\newcommand\csname conjText#1\endcsname{#2}
\begin{conj}
\label{conj:#1}
\csname conjText#1\endcsname
\end{conj}
}

\newcommand{\newQues}[2]{
  \expandafter\newcommand\csname quesText#1\endcsname{#2}
\begin{ques}
\label{ques:#1}
\csname quesText#1\endcsname
\end{ques}
}

\newcommand{\newProb}[2]{
  \expandafter\newcommand\csname probText#1\endcsname{#2}
\begin{prob}
\label{prob:#1}
\csname probText#1\endcsname
\end{prob}
}

\newcommand{\newThm}[2]{
  \expandafter\newcommand\csname thmText#1\endcsname{#2}
\begin{thm}
\label{th:#1}
\csname thmText#1\endcsname
\end{thm}
}

\newcommand{\repeatConj}[1]{
\csname theoremstyle\endcsname{plain}
\csname newtheorem\endcsname*{#1ConjRepeat}{Conjecture~\csname ref\endcsname{conj:#1}}
\csname begin\endcsname{#1ConjRepeat}
\csname conjText#1\endcsname
\csname end\endcsname{#1ConjRepeat}
}

\newcommand{\repeatQues}[1]{
\csname theoremstyle\endcsname{definition}
\csname newtheorem\endcsname*{#1QuesRepeat}{Question~\csname ref\endcsname{ques:#1}}
\csname begin\endcsname{#1QuesRepeat}
\csname quesText#1\endcsname
\csname end\endcsname{#1QuesRepeat}
}

\newcommand{\repeatProb}[1]{
\csname theoremstyle\endcsname{definition}
\csname newtheorem\endcsname*{#1ProbRepeat}{Problem~\csname ref\endcsname{prob:#1}}
\csname begin\endcsname{#1ProbRepeat}
\csname probText#1\endcsname
\csname end\endcsname{#1ProbRepeat}
}

\newcommand{\repeatThm}[1]{
\csname theoremstyle\endcsname{plain}
\csname newtheorem\endcsname*{#1ThmRepeat}{Theorem~\csname ref\endcsname{th:#1}}
\csname begin\endcsname{#1ThmRepeat}
\csname thmText#1\endcsname
\csname end\endcsname{#1ThmRepeat}
}

\title{Common Pairs of Graphs}
\author{Natalie Behague\thanks{Supported by a PIMS Postdoctoral Fellowship.} }
\author{Natasha  Morrison\thanks{Research supported by NSERC Discovery Grant RGPIN-2021-02511, NSERC Early Career Supplement DGECR-2021-00047 and a Start-Up Grant from the University of Victoria.}}
\author{Jonathan A. Noel\thanks{Research supported by NSERC Discovery Grant RGPIN-2021-02460, NSERC Early Career Supplement DGECR-2021-00024 and a Start-Up Grant from the University of Victoria.}} 

\affil{\normalsize{Department of Mathematics and Statistics, University of Victoria, Victoria, B.C., Canada.}}
\affil{\texttt{\{nbehague,nmorrison,noelj\}@uvic.ca}}

\DeclareMathOperator{\Aut}{Aut}
\DeclareMathOperator{\aut}{aut}

\DeclareTextCompositeCommand{\v}{OT1}{l}{l\nobreak\hspace{-.1em}'}
\DeclareTextCompositeCommand{\v}{OT1}{t}{t\nobreak\hspace{-.1em}'\nobreak\hspace{-.15em}}

\DeclareMathOperator{\inj}{inj}
\DeclareMathOperator{\ind}{ind}

\pgfdeclarelayer{edgelayer}
\pgfdeclarelayer{nodelayer}
\pgfsetlayers{edgelayer,nodelayer,main}

\tikzstyle{none}=[inner sep=0pt]
\definecolor{hexcolor0xf81e1c}{rgb}{0.973,0.118,0.110}
\definecolor{hexcolor0x3c00ff}{rgb}{0.235,0.000,1.000}

\tikzstyle{root}=[rectangle, fill=white,draw=black, scale=0.70]
\tikzstyle{vertex}=[circle, fill=white,draw=black, scale=0.55]
\tikzstyle{root_small}=[rectangle, fill=white,draw=black, scale=0.350]
\tikzstyle{vertex_small}=[circle, fill=white,draw=black, scale=0.35]
\tikzstyle{dashededge}=[draw=black, densely dotted]
\tikzstyle{edge}=[draw=black]

\newcommand{\ZeroZero}{
\begin{tikzpicture}[scale = 0.75]
	\begin{pgfonlayer}{nodelayer}
		\node [style=root] (2) at (330:1) {};
		\node [style=root] (1) at (90:1) {};
		\node [style=root] (0) at (210:1) {};
		\node [style=vertex] (3) at (0,0) {};
	\end{pgfonlayer}
	\begin{pgfonlayer}{edgelayer}
		\draw [style=dashededge](0.center) to (1.center);
		\draw [style=dashededge](2.center) to (1.center);
		\draw [style=dashededge](0.center) to (2.center);
		\draw [style=dashededge](3.center) to (0.center);
		\draw [style=dashededge](3.center) to (1.center);
		\draw [style=dashededge](3.center) to (2.center);
	\end{pgfonlayer}
\end{tikzpicture}}

\newcommand{\ZeroOne}{
\begin{tikzpicture}[scale = 0.75]
	\begin{pgfonlayer}{nodelayer}
		\node [style=root] (2) at (330:1) {};
		\node [style=root] (1) at (90:1) {};
		\node [style=root] (0) at (210:1) {};
		\node [style=vertex] (3) at (0,0) {};
	\end{pgfonlayer}
	\begin{pgfonlayer}{edgelayer}
		\draw [style=dashededge](0.center) to (1.center);
		\draw [style=dashededge](2.center) to (1.center);
		\draw [style=dashededge](0.center) to (2.center);
		\draw (3.center) to (0.center);
		\draw [style=dashededge](3.center) to (1.center);
		\draw [style=dashededge](3.center) to (2.center);
	\end{pgfonlayer}
\end{tikzpicture}}

\newcommand{\ZeroTwo}{
\begin{tikzpicture}[scale = 0.75]
	\begin{pgfonlayer}{nodelayer}
		\node [style=root] (2) at (330:1) {};
		\node [style=root] (1) at (90:1) {};
		\node [style=root] (0) at (210:1) {};
		\node [style=vertex] (3) at (0,0) {};
	\end{pgfonlayer}
	\begin{pgfonlayer}{edgelayer}
		\draw [style=dashededge](0.center) to (1.center);
		\draw [style=dashededge](2.center) to (1.center);
		\draw [style=dashededge](0.center) to (2.center);
		\draw [style=dashededge](3.center) to (0.center);
		\draw (3.center) to (1.center);
		\draw [style=dashededge](3.center) to (2.center);
	\end{pgfonlayer}
\end{tikzpicture}}

\newcommand{\ZeroThree}{
\begin{tikzpicture}[scale = 0.75]
	\begin{pgfonlayer}{nodelayer}
		\node [style=root] (2) at (330:1) {};
		\node [style=root] (1) at (90:1) {};
		\node [style=root] (0) at (210:1) {};
		\node [style=vertex] (3) at (0,0) {};
	\end{pgfonlayer}
	\begin{pgfonlayer}{edgelayer}
		\draw [style=dashededge](0.center) to (1.center);
		\draw [style=dashededge](2.center) to (1.center);
		\draw [style=dashededge](0.center) to (2.center);
		\draw [style=dashededge](3.center) to (0.center);
		\draw [style=dashededge](3.center) to (1.center);
		\draw (3.center) to (2.center);
	\end{pgfonlayer}
\end{tikzpicture}}

\newcommand{\ZeroFour}{
\begin{tikzpicture}[scale = 0.75]
	\begin{pgfonlayer}{nodelayer}
		\node [style=root] (2) at (330:1) {};
		\node [style=root] (1) at (90:1) {};
		\node [style=root] (0) at (210:1) {};
		\node [style=vertex] (3) at (0,0) {};
	\end{pgfonlayer}
	\begin{pgfonlayer}{edgelayer}
		\draw [style=dashededge](0.center) to (1.center);
		\draw [style=dashededge](2.center) to (1.center);
		\draw [style=dashededge](0.center) to (2.center);
		\draw (3.center) to (0.center);
		\draw (3.center) to (1.center);
		\draw [style=dashededge](3.center) to (2.center);
	\end{pgfonlayer}
\end{tikzpicture}}

\newcommand{\ZeroFive}{
\begin{tikzpicture}[scale = 0.75]
	\begin{pgfonlayer}{nodelayer}
		\node [style=root] (2) at (330:1) {};
		\node [style=root] (1) at (90:1) {};
		\node [style=root] (0) at (210:1) {};
		\node [style=vertex] (3) at (0,0) {};
	\end{pgfonlayer}
	\begin{pgfonlayer}{edgelayer}
		\draw [style=dashededge](0.center) to (1.center);
		\draw [style=dashededge](2.center) to (1.center);
		\draw [style=dashededge](0.center) to (2.center);
		\draw (3.center) to (0.center);
		\draw [style=dashededge](3.center) to (1.center);
		\draw (3.center) to (2.center);
	\end{pgfonlayer}
\end{tikzpicture}}

\newcommand{\ZeroSix}{
\begin{tikzpicture}[scale = 0.75]
	\begin{pgfonlayer}{nodelayer}
		\node [style=root] (2) at (330:1) {};
		\node [style=root] (1) at (90:1) {};
		\node [style=root] (0) at (210:1) {};
		\node [style=vertex] (3) at (0,0) {};
	\end{pgfonlayer}
	\begin{pgfonlayer}{edgelayer}
		\draw [style=dashededge](0.center) to (1.center);
		\draw [style=dashededge](2.center) to (1.center);
		\draw [style=dashededge](0.center) to (2.center);
		\draw [style=dashededge](3.center) to (0.center);
		\draw (3.center) to (1.center);
		\draw (3.center) to (2.center);
	\end{pgfonlayer}
\end{tikzpicture}}

\newcommand{\ZeroSeven}{
\begin{tikzpicture}[scale = 0.75]
	\begin{pgfonlayer}{nodelayer}
		\node [style=root] (2) at (330:1) {};
		\node [style=root] (1) at (90:1) {};
		\node [style=root] (0) at (210:1) {};
		\node [style=vertex] (3) at (0,0) {};
	\end{pgfonlayer}
	\begin{pgfonlayer}{edgelayer}
		\draw [style=dashededge](0.center) to (1.center);
		\draw [style=dashededge](2.center) to (1.center);
		\draw [style=dashededge](0.center) to (2.center);
		\draw (3.center) to (0.center);
		\draw (3.center) to (1.center);
		\draw (3.center) to (2.center);
	\end{pgfonlayer}
\end{tikzpicture}}

\newcommand{\OneZeroZero}{
\begin{tikzpicture}[scale = 0.75]
	\begin{pgfonlayer}{nodelayer}
		\node [style=root] (2) at (330:1) {};
		\node [style=root] (1) at (90:1) {};
		\node [style=root] (0) at (210:1) {};
		\node [style=vertex] (3) at (0,0) {};
	\end{pgfonlayer}
	\begin{pgfonlayer}{edgelayer}
		\draw (0.center) to (1.center);
		\draw [style=dashededge](2.center) to (1.center);
		\draw [style=dashededge](0.center) to (2.center);
		\draw [style=dashededge](3.center) to (0.center);
		\draw [style=dashededge](3.center) to (1.center);
		\draw [style=dashededge](3.center) to (2.center);
	\end{pgfonlayer}
\end{tikzpicture}}

\newcommand{\OneZeroOne}{
\begin{tikzpicture}[scale = 0.75]
	\begin{pgfonlayer}{nodelayer}
		\node [style=root] (2) at (330:1) {};
		\node [style=root] (1) at (90:1) {};
		\node [style=root] (0) at (210:1) {};
		\node [style=vertex] (3) at (0,0) {};
	\end{pgfonlayer}
	\begin{pgfonlayer}{edgelayer}
		\draw (0.center) to (1.center);
		\draw [style=dashededge](2.center) to (1.center);
		\draw [style=dashededge](0.center) to (2.center);
		\draw (3.center) to (0.center);
		\draw [style=dashededge](3.center) to (1.center);
		\draw [style=dashededge](3.center) to (2.center);
	\end{pgfonlayer}
\end{tikzpicture}}

\newcommand{\OneZeroTwo}{
\begin{tikzpicture}[scale = 0.75]
	\begin{pgfonlayer}{nodelayer}
		\node [style=root] (2) at (330:1) {};
		\node [style=root] (1) at (90:1) {};
		\node [style=root] (0) at (210:1) {};
		\node [style=vertex] (3) at (0,0) {};
	\end{pgfonlayer}
	\begin{pgfonlayer}{edgelayer}
		\draw (0.center) to (1.center);
		\draw [style=dashededge](2.center) to (1.center);
		\draw [style=dashededge](0.center) to (2.center);
		\draw [style=dashededge](3.center) to (0.center);
		\draw (3.center) to (1.center);
		\draw [style=dashededge](3.center) to (2.center);
	\end{pgfonlayer}
\end{tikzpicture}}

\newcommand{\OneZeroThree}{
\begin{tikzpicture}[scale = 0.75]
	\begin{pgfonlayer}{nodelayer}
		\node [style=root] (2) at (330:1) {};
		\node [style=root] (1) at (90:1) {};
		\node [style=root] (0) at (210:1) {};
		\node [style=vertex] (3) at (0,0) {};
	\end{pgfonlayer}
	\begin{pgfonlayer}{edgelayer}
		\draw (0.center) to (1.center);
		\draw [style=dashededge](2.center) to (1.center);
		\draw [style=dashededge](0.center) to (2.center);
		\draw [style=dashededge](3.center) to (0.center);
		\draw [style=dashededge](3.center) to (1.center);
		\draw (3.center) to (2.center);
	\end{pgfonlayer}
\end{tikzpicture}}

\newcommand{\OneZeroFour}{
\begin{tikzpicture}[scale = 0.75]
	\begin{pgfonlayer}{nodelayer}
		\node [style=root] (2) at (330:1) {};
		\node [style=root] (1) at (90:1) {};
		\node [style=root] (0) at (210:1) {};
		\node [style=vertex] (3) at (0,0) {};
	\end{pgfonlayer}
	\begin{pgfonlayer}{edgelayer}
		\draw (0.center) to (1.center);
		\draw [style=dashededge](2.center) to (1.center);
		\draw [style=dashededge](0.center) to (2.center);
		\draw (3.center) to (0.center);
		\draw (3.center) to (1.center);
		\draw [style=dashededge](3.center) to (2.center);
	\end{pgfonlayer}
\end{tikzpicture}}

\newcommand{\OneZeroFive}{
\begin{tikzpicture}[scale = 0.75]
	\begin{pgfonlayer}{nodelayer}
		\node [style=root] (2) at (330:1) {};
		\node [style=root] (1) at (90:1) {};
		\node [style=root] (0) at (210:1) {};
		\node [style=vertex] (3) at (0,0) {};
	\end{pgfonlayer}
	\begin{pgfonlayer}{edgelayer}
		\draw (0.center) to (1.center);
		\draw [style=dashededge](2.center) to (1.center);
		\draw [style=dashededge](0.center) to (2.center);
		\draw (3.center) to (0.center);
		\draw [style=dashededge](3.center) to (1.center);
		\draw (3.center) to (2.center);
	\end{pgfonlayer}
\end{tikzpicture}}

\newcommand{\OneZeroSix}{
\begin{tikzpicture}[scale = 0.75]
	\begin{pgfonlayer}{nodelayer}
		\node [style=root] (2) at (330:1) {};
		\node [style=root] (1) at (90:1) {};
		\node [style=root] (0) at (210:1) {};
		\node [style=vertex] (3) at (0,0) {};
	\end{pgfonlayer}
	\begin{pgfonlayer}{edgelayer}
		\draw (0.center) to (1.center);
		\draw [style=dashededge](2.center) to (1.center);
		\draw [style=dashededge](0.center) to (2.center);
		\draw [style=dashededge](3.center) to (0.center);
		\draw (3.center) to (1.center);
		\draw (3.center) to (2.center);
	\end{pgfonlayer}
\end{tikzpicture}}

\newcommand{\OneZeroSeven}{
\begin{tikzpicture}[scale = 0.75]
	\begin{pgfonlayer}{nodelayer}
		\node [style=root] (2) at (330:1) {};
		\node [style=root] (1) at (90:1) {};
		\node [style=root] (0) at (210:1) {};
		\node [style=vertex] (3) at (0,0) {};
	\end{pgfonlayer}
	\begin{pgfonlayer}{edgelayer}
		\draw (0.center) to (1.center);
		\draw [style=dashededge](2.center) to (1.center);
		\draw [style=dashededge](0.center) to (2.center);
		\draw (3.center) to (0.center);
		\draw (3.center) to (1.center);
		\draw (3.center) to (2.center);
	\end{pgfonlayer}
\end{tikzpicture}}

\newcommand{\TwoZeroZero}{
\begin{tikzpicture}[scale = 0.75]
	\begin{pgfonlayer}{nodelayer}
		\node [style=root] (2) at (330:1) {};
		\node [style=root] (1) at (90:1) {};
		\node [style=root] (0) at (210:1) {};
		\node [style=vertex] (3) at (0,0) {};
	\end{pgfonlayer}
	\begin{pgfonlayer}{edgelayer}
		\draw (0.center) to (1.center);
		\draw [style=dashededge](2.center) to (1.center);
		\draw (0.center) to (2.center);
		\draw [style=dashededge](3.center) to (0.center);
		\draw [style=dashededge](3.center) to (1.center);
		\draw [style=dashededge](3.center) to (2.center);
	\end{pgfonlayer}
\end{tikzpicture}}

\newcommand{\TwoZeroOne}{
\begin{tikzpicture}[scale = 0.75]
	\begin{pgfonlayer}{nodelayer}
		\node [style=root] (2) at (330:1) {};
		\node [style=root] (1) at (90:1) {};
		\node [style=root] (0) at (210:1) {};
		\node [style=vertex] (3) at (0,0) {};
	\end{pgfonlayer}
	\begin{pgfonlayer}{edgelayer}
		\draw (0.center) to (1.center);
		\draw [style=dashededge](2.center) to (1.center);
		\draw (0.center) to (2.center);
		\draw (3.center) to (0.center);
		\draw [style=dashededge](3.center) to (1.center);
		\draw [style=dashededge](3.center) to (2.center);
	\end{pgfonlayer}
\end{tikzpicture}}

\newcommand{\TwoZeroTwo}{
\begin{tikzpicture}[scale = 0.75]
	\begin{pgfonlayer}{nodelayer}
		\node [style=root] (2) at (330:1) {};
		\node [style=root] (1) at (90:1) {};
		\node [style=root] (0) at (210:1) {};
		\node [style=vertex] (3) at (0,0) {};
	\end{pgfonlayer}
	\begin{pgfonlayer}{edgelayer}
		\draw (0.center) to (1.center);
		\draw [style=dashededge](2.center) to (1.center);
		\draw (0.center) to (2.center);
		\draw [style=dashededge](3.center) to (0.center);
		\draw (3.center) to (1.center);
		\draw [style=dashededge](3.center) to (2.center);
	\end{pgfonlayer}
\end{tikzpicture}}

\newcommand{\TwoZeroThree}{
\begin{tikzpicture}[scale = 0.75]
	\begin{pgfonlayer}{nodelayer}
		\node [style=root] (2) at (330:1) {};
		\node [style=root] (1) at (90:1) {};
		\node [style=root] (0) at (210:1) {};
		\node [style=vertex] (3) at (0,0) {};
	\end{pgfonlayer}
	\begin{pgfonlayer}{edgelayer}
		\draw (0.center) to (1.center);
		\draw [style=dashededge](2.center) to (1.center);
		\draw (0.center) to (2.center);
		\draw [style=dashededge](3.center) to (0.center);
		\draw [style=dashededge](3.center) to (1.center);
		\draw (3.center) to (2.center);
	\end{pgfonlayer}
\end{tikzpicture}}

\newcommand{\TwoZeroFour}{
\begin{tikzpicture}[scale = 0.75]
	\begin{pgfonlayer}{nodelayer}
		\node [style=root] (2) at (330:1) {};
		\node [style=root] (1) at (90:1) {};
		\node [style=root] (0) at (210:1) {};
		\node [style=vertex] (3) at (0,0) {};
	\end{pgfonlayer}
	\begin{pgfonlayer}{edgelayer}
		\draw (0.center) to (1.center);
		\draw [style=dashededge](2.center) to (1.center);
		\draw (0.center) to (2.center);
		\draw (3.center) to (0.center);
		\draw (3.center) to (1.center);
		\draw [style=dashededge](3.center) to (2.center);
	\end{pgfonlayer}
\end{tikzpicture}}

\newcommand{\TwoZeroFive}{
\begin{tikzpicture}[scale = 0.75]
	\begin{pgfonlayer}{nodelayer}
		\node [style=root] (2) at (330:1) {};
		\node [style=root] (1) at (90:1) {};
		\node [style=root] (0) at (210:1) {};
		\node [style=vertex] (3) at (0,0) {};
	\end{pgfonlayer}
	\begin{pgfonlayer}{edgelayer}
		\draw (0.center) to (1.center);
		\draw [style=dashededge](2.center) to (1.center);
		\draw (0.center) to (2.center);
		\draw (3.center) to (0.center);
		\draw [style=dashededge](3.center) to (1.center);
		\draw (3.center) to (2.center);
	\end{pgfonlayer}
\end{tikzpicture}}

\newcommand{\TwoZeroSix}{
\begin{tikzpicture}[scale = 0.75]
	\begin{pgfonlayer}{nodelayer}
		\node [style=root] (2) at (330:1) {};
		\node [style=root] (1) at (90:1) {};
		\node [style=root] (0) at (210:1) {};
		\node [style=vertex] (3) at (0,0) {};
	\end{pgfonlayer}
	\begin{pgfonlayer}{edgelayer}
		\draw (0.center) to (1.center);
		\draw [style=dashededge](2.center) to (1.center);
		\draw (0.center) to (2.center);
		\draw [style=dashededge](3.center) to (0.center);
		\draw (3.center) to (1.center);
		\draw (3.center) to (2.center);
	\end{pgfonlayer}
\end{tikzpicture}}

\newcommand{\TwoZeroSeven}{
\begin{tikzpicture}[scale = 0.75]
	\begin{pgfonlayer}{nodelayer}
		\node [style=root] (2) at (330:1) {};
		\node [style=root] (1) at (90:1) {};
		\node [style=root] (0) at (210:1) {};
		\node [style=vertex] (3) at (0,0) {};
	\end{pgfonlayer}
	\begin{pgfonlayer}{edgelayer}
		\draw (0.center) to (1.center);
		\draw [style=dashededge](2.center) to (1.center);
		\draw (0.center) to (2.center);
		\draw (3.center) to (0.center);
		\draw (3.center) to (1.center);
		\draw (3.center) to (2.center);
	\end{pgfonlayer}
\end{tikzpicture}}

\newcommand{\ThreeZeroZero}{
\begin{tikzpicture}[scale = 0.75]
	\begin{pgfonlayer}{nodelayer}
		\node [style=root] (2) at (330:1) {};
		\node [style=root] (1) at (90:1) {};
		\node [style=root] (0) at (210:1) {};
		\node [style=vertex] (3) at (0,0) {};
	\end{pgfonlayer}
	\begin{pgfonlayer}{edgelayer}
		\draw (0.center) to (1.center);
		\draw (2.center) to (1.center);
		\draw (0.center) to (2.center);
		\draw [style=dashededge](3.center) to (0.center);
		\draw [style=dashededge](3.center) to (1.center);
		\draw [style=dashededge](3.center) to (2.center);
	\end{pgfonlayer}
\end{tikzpicture}}

\newcommand{\ThreeZeroOne}{
\begin{tikzpicture}[scale = 0.75]
	\begin{pgfonlayer}{nodelayer}
		\node [style=root] (2) at (330:1) {};
		\node [style=root] (1) at (90:1) {};
		\node [style=root] (0) at (210:1) {};
		\node [style=vertex] (3) at (0,0) {};
	\end{pgfonlayer}
	\begin{pgfonlayer}{edgelayer}
		\draw (0.center) to (1.center);
		\draw (2.center) to (1.center);
		\draw (0.center) to (2.center);
		\draw (3.center) to (0.center);
		\draw [style=dashededge](3.center) to (1.center);
		\draw [style=dashededge](3.center) to (2.center);
	\end{pgfonlayer}
\end{tikzpicture}}

\newcommand{\ThreeZeroTwo}{
\begin{tikzpicture}[scale = 0.75]
	\begin{pgfonlayer}{nodelayer}
		\node [style=root] (2) at (330:1) {};
		\node [style=root] (1) at (90:1) {};
		\node [style=root] (0) at (210:1) {};
		\node [style=vertex] (3) at (0,0) {};
	\end{pgfonlayer}
	\begin{pgfonlayer}{edgelayer}
		\draw (0.center) to (1.center);
		\draw (2.center) to (1.center);
		\draw (0.center) to (2.center);
		\draw [style=dashededge](3.center) to (0.center);
		\draw (3.center) to (1.center);
		\draw [style=dashededge](3.center) to (2.center);
	\end{pgfonlayer}
\end{tikzpicture}}

\newcommand{\ThreeZeroThree}{
\begin{tikzpicture}[scale = 0.75]
	\begin{pgfonlayer}{nodelayer}
		\node [style=root] (2) at (330:1) {};
		\node [style=root] (1) at (90:1) {};
		\node [style=root] (0) at (210:1) {};
		\node [style=vertex] (3) at (0,0) {};
	\end{pgfonlayer}
	\begin{pgfonlayer}{edgelayer}
		\draw (0.center) to (1.center);
		\draw (2.center) to (1.center);
		\draw (0.center) to (2.center);
		\draw [style=dashededge](3.center) to (0.center);
		\draw [style=dashededge](3.center) to (1.center);
		\draw (3.center) to (2.center);
	\end{pgfonlayer}
\end{tikzpicture}}

\newcommand{\ThreeZeroFour}{
\begin{tikzpicture}[scale = 0.75]
	\begin{pgfonlayer}{nodelayer}
		\node [style=root] (2) at (330:1) {};
		\node [style=root] (1) at (90:1) {};
		\node [style=root] (0) at (210:1) {};
		\node [style=vertex] (3) at (0,0) {};
	\end{pgfonlayer}
	\begin{pgfonlayer}{edgelayer}
		\draw (0.center) to (1.center);
		\draw (2.center) to (1.center);
		\draw (0.center) to (2.center);
		\draw (3.center) to (0.center);
		\draw (3.center) to (1.center);
		\draw [style=dashededge](3.center) to (2.center);
	\end{pgfonlayer}
\end{tikzpicture}}

\newcommand{\ThreeZeroFive}{
\begin{tikzpicture}[scale = 0.75]
	\begin{pgfonlayer}{nodelayer}
		\node [style=root] (2) at (330:1) {};
		\node [style=root] (1) at (90:1) {};
		\node [style=root] (0) at (210:1) {};
		\node [style=vertex] (3) at (0,0) {};
	\end{pgfonlayer}
	\begin{pgfonlayer}{edgelayer}
		\draw (0.center) to (1.center);
		\draw (2.center) to (1.center);
		\draw (0.center) to (2.center);
		\draw (3.center) to (0.center);
		\draw [style=dashededge](3.center) to (1.center);
		\draw (3.center) to (2.center);
	\end{pgfonlayer}
\end{tikzpicture}}

\newcommand{\ThreeZeroSix}{
\begin{tikzpicture}[scale = 0.75]
	\begin{pgfonlayer}{nodelayer}
		\node [style=root] (2) at (330:1) {};
		\node [style=root] (1) at (90:1) {};
		\node [style=root] (0) at (210:1) {};
		\node [style=vertex] (3) at (0,0) {};
	\end{pgfonlayer}
	\begin{pgfonlayer}{edgelayer}
		\draw (0.center) to (1.center);
		\draw (2.center) to (1.center);
		\draw (0.center) to (2.center);
		\draw [style=dashededge](3.center) to (0.center);
		\draw (3.center) to (1.center);
		\draw (3.center) to (2.center);
	\end{pgfonlayer}
\end{tikzpicture}}

\newcommand{\ThreeZeroSeven}{
\begin{tikzpicture}[scale = 0.75]
	\begin{pgfonlayer}{nodelayer}
		\node [style=root] (2) at (330:1) {};
		\node [style=root] (1) at (90:1) {};
		\node [style=root] (0) at (210:1) {};
		\node [style=vertex] (3) at (0,0) {};
	\end{pgfonlayer}
	\begin{pgfonlayer}{edgelayer}
		\draw (0.center) to (1.center);
		\draw (2.center) to (1.center);
		\draw (0.center) to (2.center);
		\draw (3.center) to (0.center);
		\draw (3.center) to (1.center);
		\draw (3.center) to (2.center);
	\end{pgfonlayer}
\end{tikzpicture}}

\begin{document}

\maketitle

\begin{abstract}
A graph $H$ is said to be \emph{common} if the number of monochromatic labelled copies of $H$ in a red/blue edge colouring of a large complete graph is asymptotically minimized by a random colouring in which each edge is equally likely to be red or blue. We extend this notion to an off-diagonal setting. That is, we define a pair $(H_1,H_2)$ of graphs to be $(p,1-p)$-common if a particular linear combination of the density of $H_1$ in red and $H_2$ in blue is asymptotically minimized by a random colouring in which each edge is coloured red with probability $p$ and blue with probability $1-p$. Our results include off-diagonal extensions of several standard theorems on common graphs and novel results for common pairs of graphs with no natural analogue in the classical setting. 
\end{abstract}

\section{Introduction}

Ramsey's Theorem~\cite{Ramsey29} implies that, for every graph $H$, there exists an integer $N$ such that every colouring of the edges of the complete graph $K_N$ with red and blue contains a monochromatic copy of $H$. The minimum such $N$ is known as the \emph{Ramsey number} of $H$. The problem of estimating Ramsey numbers is widely studied and notoriously difficult; see, e.g., the survey of Conlon, Fox and Sudakov~\cite{ConlonFoxSudakov15}. 

The closely related \emph{Ramsey multiplicity problem} asks for the asymptotics of the minimum possible number of monochromatic copies of $H$ in a $2$-edge colouring of $K_N$ as $N$ tends to infinity~\cite{Fox08,Conlon12,FoxWidgerson22}. While Ramsey multiplicity problems typically focus on counting monochromatic copies of a single graph, we consider the problem of minimizing the number of red copies of a graph $H_1$ plus blue copies of another graph $H_2$, each of which is appropriately normalized and weighted. Asymmetric Ramsey multiplicity problems of a similar flavour for cliques were recently investigated by Parczyk, Pokutta, Spiegel, and Szab\'o~\cite{Parczyk+22+}. Our focus in this paper is on classifying pairs of graphs for which the minimum is attained by a random colouring in which each colour appears with a prescribed density. Such problems have a rich history starting with the beautiful result of Goodman~\cite[Theorem~1]{Goodman59} which says that the number of monochromatic triangles in a 2-edge colouring of a large complete graph is asymptotically minimized by an unbiased random colouring. 

In order to state our results, we require some notation and terminology. A \emph{kernel} is a bounded measurable function $U:[0,1]^2\to \mathbb{R}$ which is \emph{symmetric}, meaning that $U(x,y)=U(y,x)$ for all $x,y\in [0,1]$. The \emph{homomorphism density} of a graph $H$ in a kernel $U$ is 
\[t(H,U) := \int\displaylimits_{[0,1]^{V(H)}}\prod_{uv\in E(H)}U(x_u,x_v)dx_{V(H)}\]
where $x_{V(H)}$ is a vector $(x_u: u\in V(H))$ of variables indexed by $V(H)$. A \emph{graphon} is a kernel $W$ such that $0\leq W(x,y)\leq 1$ for all $x,y\in[0,1]$. Informally, a graphon can be thought of as a complete graph with vertex set $[0,1]$ such that, for $x,y\in [0,1]$, the edge $xy$ is assigned a weight of $W(x,y)$. A graph $H$ is said to be \emph{common} if
\[t(H,W) + t(H,1-W)\geq (1/2)^{e(H)-1}\]
for every graphon $W$, where $e(H):=|E(H)|$. In other words, $H$ is common if $t(H,W)+t(H,1-W)$ is minimized when $W$ is the constant graphon $W=1/2$; this graphon represents the ``limit'' of a sequence of Erd\H{o}s--R\'enyi random graphs $G(n,1/2)$ as $n\to\infty$. The aforementioned result of Goodman~\cite{Goodman59} on monochromatic triangles implies that $K_3$ is common. 

The idea to investigate off-diagonal variants of the notion of common graphs first came to our attention when it was proposed by Hladk\'y after a webinar talk given by the third author~\cite{Noel20}. Specifically, Hladk\'y's question focused on pairs of graphs with the same number of edges and random colourings with the same density of each colour. In this paper, we consider the more general situation in which the two graphs can have different numbers of edges and the optimal colouring is random with different proportions of red and blue edges. 

\begin{defn}\label{defn:commonPair}
For $p_1,p_2\in (0,1)$ such that $p_1+p_2=1$, a pair $(H_1,H_2)$ of graphs is said to be \emph{$(p_1,p_2)$-common} if
\[\frac{t(H_1,W_1)}{e(H_1)p_1^{e(H_1)-1}} + \frac{t(H_2,W_2)}{e(H_2)p_2^{e(H_2)-1}}\geq \frac{p_1}{e(H_1)} + \frac{p_2}{e(H_2)}\]
for any graphons $W_1$ and $W_2$ such that $W_1+W_2=1$.
\end{defn} 

By plugging $p_1=p_2=1/2$ and $H_1=H_2=H$ into Definition~\ref{defn:commonPair}, one gets that a graph $H$ is common if and only if the pair $(H,H)$ is $(1/2,1/2)$-common; thus Definition~\ref{defn:commonPair} generalizes the notion of common graphs.


Our focus in this paper is on extending classical results for common graphs to this off-diagonal setting and obtaining interesting examples of $(p,1-p)$-common pairs of graphs. These results will shed light on the set
\[\pi(H_1,H_2):=\{p\in (0,1): (H_1,H_2)\text{ is }(p,1-p)\text{-common}\}\]
for two graphs $H_1$ and $H_2$. In particular, the results below imply that $\pi(H_1,H_2)$ can be equal to the full interval $(0,1)$ (by Theorem~\ref{th:Sid}), a singleton (by Corollary~\ref{cor:Girth}), $\emptyset$ (by Theorem~\ref{th:KFour}) or a proper subset of $(0,1)$ containing at least two elements (by Theorems~\ref{th:Flag} and~\ref{th:CFourCFiveNot}). We conjecture that $\pi(H_1,H_2)$ is always an interval.

\newConj{Interval}{For any graphs $H_1$ and $H_2$, the set $\pi(H_1,H_2)$ is an interval.}

The problem of classifying common graphs is intimately linked to Sidorenko's Conjecture~\cite{Sidorenko93}, which says that, if $H$ is bipartite, then
\begin{equation}\label{eq:Sid}t(H,W)\geq t(K_2,W)^{e(H)}\end{equation}
for every graphon $W$. If $H$ satisfies \eqref{eq:Sid} for every $W$, then we say that $H$ is \emph{Sidorenko}. A simple convexity argument shows that, if $H$ is Sidorenko, then $H$ is common. Therefore, all of the positive results on Sidorenko's Conjecture automatically yield examples of common graphs. Despite a great deal of recent progress~\cite{ConlonKimLeeLee18,ConlonLee21,KimLeeLee16,ConlonLee17,ConlonFoxSudakov10,Szegedy15,Hatami10}, Sidorenko's Conjecture remains open. Interestingly, there is currently no example of a bipartite graph which is known to be common and not known to be Sidorenko. Our first result says that pairs of Sidorenko graphs are $(p,1-p)$-common for all $p\in(0,1)$. We also propose converses to two consequences of this result.

\newThm{Sid}{If $H_1$ and $H_2$ are Sidorenko, then $(H_1,H_2)$ is $(p_1,p_2)$-common for any positive reals $p_1$ and $p_2$ such that $p_1+p_2=1$.}

\newQues{Sid}{Suppose that $H$ is a graph such that $(H,H)$ is $(p,1-p)$-common for all $p\in (0,1)$. Does it follow that $H$ is Sidorenko? What if $(H,H)$ is $(p,1-p)$-common for \emph{some} $p\in (0,1)\setminus\{1/2\}$?}

Regarding non-bipartite graphs, our next result provides an interesting necessary condition on odd cycles in $(p_1,p_2)$-common pairs of graphs. Given a graph $H$ and integer $k\geq3$, we let $c_k(H)$ be the number of unlabelled cycles of length $k$ in $H$. The \emph{girth} of $H$ is $g(H):=\min\{k: c_k(H)\geq1\}$. In particular,  $g(H)=\infty$ if $H$ has no cycles.  

\newThm{Girth}{Let $H_1$ and $H_2$ be graphs and let $p_1$ and $p_2$ be positive reals such that $p_1 + p_2 = 1$. Suppose that $g(H_2)$ is odd. If $(H_1,H_2)$ is $(p_1,p_2)$-common, then either $g(H_1)<g(H_2)$ or $g(H_1)=g(H_2)=k$ and
\[\frac{c_k(H_1)}{e(H_1)p_1^{k-1}} = \frac{c_k(H_2)}{e(H_2)p_2^{k-1}}.\]}

Sidorenko~\cite[Corollary~1]{Sidorenko89} proved that odd cycles are common; thus, Theorem~\ref{th:Girth} yields the following corollary.

\begin{cor}
\label{cor:Girth}
If $k\geq3$ is odd, then $(C_k,C_k)$ is $(p,1-p)$-common if and only if $p=1/2$. 
\end{cor}

By Theorem~\ref{th:Girth}, if $H_1$ and $H_2$ have odd girth, then there is at most one $p\in (0,1)$ such that $(H_1,H_2)$ is $(p,1-p)$-common. The next theorem provides an example of a pair $(H_1,H_2)$ of graphs, one of which has odd girth, such that $(H_1,H_2)$ is $(p,1-p)$-common for more than one $p\in (0,1)$.

\newThm{Flag}{$(C_4,C_5)$ is $(1/2,1/2)$-common and $(1/3,2/3)$-common.}

\newThm{CFourCFiveNot}{If $p > 0.518$, then $(C_4,C_5)$ is not $(p,1-p)$-common.}

The previous two theorems suggest the following open problem. 

\newProb{CFourCFive}{Determine $\pi(C_4,C_5)$.}

In 1962, Erd\H{o}s~\cite[Equation~(3)]{Erdos62} conjectured that all complete graphs are common and, in 1980, Burr and Rosta~\cite{BurrRosta80} extended Erd\H{o}s' conjecture to the statement that every graph $H$ is common. Both of these conjectures turn out to be false; Sidorenko~\cite[p.~881]{Sidorenko89} showed that a triangle with a pendant edge is uncommon and Thomason~\cite{Thomason89} proved that $K_4$ is uncommon. The latter result was extended by Jagger, \v{S}\v{t}ov\'{\i}\v{c}ek and Thomason~\cite[Theorem~12]{JaggerStovicekThomason96} to the very general statement that any graph containing $K_4$ is uncommon. We generalize this to the off-diagonal setting in a strong sense. 

\newThm{KFour}{Let $H_1$ and $H_2$ be graphs and let $p_1,p_2\in (0,1)$ so that $p_1+p_2=1$. If $(H_1,H_2)$ is $(p_1,p_2)$-common, then both of $H_1$ and $H_2$ are $K_4$-free.}

We conjecture that the existence of a $K_4$ subgraph is the only barrier to a graph being contained in a $(p,1-p)$-common pair for some $p\in(0,1)$. We note that it would also be interesting to prove more restricted versions of the following conjecture; in particular, proving it for bipartite graphs could be viewed as a natural weak form of Sidorenko's Conjecture. 

\newConj{KFourBarrier}{If $H_1$ is $K_4$-free, then there exists a graph $H_2$ and $p\in (0,1)$ such that $(H_1,H_2)$ is $(p,1-p)$-common.}

In light of the fact that common graphs cannot contain a $K_4$, a central problem in the area has been to determine whether common graphs can have high chromatic number. The first example of a non-3-colourable common graph, the $5$-wheel, was found in~\cite{Hatami+12}. In a recent breakthrough paper, Kr\'a\v{l}, Volec and Wei~\cite{KralVolecWei22+} proved that there are common graphs of arbitrary chromatic number; see also~\cite{KoLee22}.

In Section~\ref{sec:Sid}, we prove Theorem~\ref{th:Sid} and a result regarding disjoint unions of $(p_1,p_2)$-common pairs of graphs. Then, in Section~\ref{sec:odd}, we recall basic facts about kernels, generalize some standard lemmas from the study of common graphs and use these statements to prove Theorem~\ref{th:Girth}. We adapt a construction of Jagger, \v{S}\v{t}ov\'i\v{c}ek and Thomason~\cite{JaggerStovicekThomason96} in Section~\ref{sec:K4} to prove Theorem~\ref{th:KFour}. In Section~\ref{sec:flags}, we apply the powerful method of flag algebras to establish Theorem~\ref{th:Flag} and provide a simple construction of a graphon which proves Theorem~\ref{th:CFourCFiveNot}. Some of the tedious calculations needed to verify the proof of Theorem~\ref{th:Flag} have been relegated to appendices which have been included in an ancilliary file with the arxiv submission of the paper: \url{https://arxiv.org/src/2208.02045/anc/commonExtraAppendices.pdf}. We conclude the paper in Sections~\ref{sec:multicol} and~\ref{sec:concl} by discussing a multicolour generalization and proposing several open problems. Further examples of common pairs of graphs constructed from odd cycles via certain ``gluing operations,'' proven using ideas from information theory, can be found in our companion paper~\cite{SecondPaper}. 

\section{Sidorenko Graphs and Disjoint Unions}
\label{sec:Sid}

We start by proving Theorem~\ref{th:Sid}, which we restate below for convenience. The proof involves an application of the following standard inequality. 

\begin{lem}[Bernoulli's Inequality, see~{\cite[Section~2.4]{Mitrinovic70}}]
\label{lem:Bernoulli}
If $1+x\geq 0$ and $r\geq 1$, then
\[(1+x)^r\geq 1+rx.\]
\end{lem}

\repeatThm{Sid}

\begin{proof}
Let $p_1,p_2\in (0,1)$ such that $p_1+p_2=1$ and let $W_1$ and $W_2$ be graphons such that $W_1+W_2=1$. Define $x:=t(K_2,W_1)-p_1$. By definition of $x$, we have $t(K_2,W_1)=p_1+x$ and $t(K_2,W_2)=p_2-x$. Since $t(K_2,W_1),t(K_2,W_2)\geq0$, we have $-p_1\leq x\leq p_2$. Since $H_1$ and $H_2$ are Sidorenko,
\[\frac{t(H_1,W_1)}{e(H_1)p_1^{e(H_1)-1}} + \frac{t(H_2,W_2)}{e(H_2)p_2^{e(H_2)-1}} \geq \frac{(p_1+x)^{e(H_1)}}{e(H_1)p_1^{e(H_1)-1}} + \frac{(p_2-x)^{e(H_2)}}{e(H_2)p_2^{e(H_2)-1}}\]
\[=\frac{p_1(1+x/p_1)^{e(H_1)}}{e(H_1)} + \frac{p_2(1-x/p_2)^{e(H_2)}}{e(H_2)}.\]
By applying Bernoulli's Inequality (Theorem~\ref{lem:Bernoulli}) to both terms, we get that this is at least $\frac{p_1}{e(H_1)}+\frac{p_2}{e(H_2)}$, which completes the proof. 
\end{proof}

Given two graphs $F_1$ and $F_2$, let $F_1\sqcup F_2$ denote the disjoint union of $F_1$ and $F_2$. Next, we investigate the effect of taking disjoint unions on the property of being $(p_1,p_2)$-common.

\begin{thm}
\label{th:unions}
Let $t\geq1$, let $F_1,\dots,F_s$ and $H$ be non-empty graphs, let $F:=\bigsqcup_{i=1}^s F_i$ and let $p_1$ and $p_2$ be positive reals such that $p_1+p_2=1$. If $e(F_1)=\cdots =e(F_s)$ and all of the pairs $(F_i,H)$ for $1\leq i\leq s$ are $(p_1,p_2)$-common, then $(F,H)$ is $(p_1,p_2)$-common. 
\end{thm}

\begin{proof}
Define $k=e(F_1)$ and let $W_1$ and $W_2$ be graphons such that $W_1+W_2=1$. Define
\[x:=-\frac{p_1}{k} +\min_{1\leq i\leq t}\left\{\frac{t(F_i,W_1)}{kp_1^{k-1}}\right\}.\]
For any index $1\leq i\leq s$, we have
\begin{equation}\label{eq:Fi}t(F_i,W_1) \geq \min_{1\leq i\leq s}\left\{t(F_i,W_1)\right\} = kp_1^{k-1}\left(\frac{p_1}{k}+x\right)\geq0.\end{equation}
Since $(F_i,H)$ is $(p_1,p_2)$-common for all $1\leq i\leq s$ and all of the graphs $F_1,\dots,F_s$ have $k$ edges, we have
\begin{equation}\label{eq:F3}\frac{t(H,W_2)}{e(H)p_2^{e(H)-1}}\geq\frac{p_1}{k} +  \frac{p_2}{e(H)}-\min_{1\leq i\leq s}\left\{\frac{t(F_i,W_1)}{kp_1^{k-1}}\right\} = \frac{p_2}{e(H)}-x.\end{equation}

Note that $t(F,W_1) = \prod_{i=1}^s t(F_i,W_1)$. So, combining \eqref{eq:Fi} and \eqref{eq:F3} yields
\[\frac{t(F,W_1)}{skp_1^{sk-1}} + \frac{t(H,W_2)}{e(H)p_2^{e(H)-1}}\geq\frac{k^sp_1^{s(k-1)}(p_1/k+x)^s}{skp_1^{sk-1}} + \frac{p_2}{e(H)}-x=\frac{p_1(1+xk/p_1)^s}{sk} + \frac{p_2}{e(H)}-x.\]
By \eqref{eq:Fi}, we can apply Bernoulli's Inequality (Theorem~\ref{lem:Bernoulli}) to get
\[\frac{p_1(1+xk/p_1)^s}{sk} + \frac{p_2}{e(H)}-x\geq \frac{p_1(1+xsk/p_1)}{sk} + \frac{p_2}{e(H)}-x\]
\[=\frac{p_1}{sk} + \frac{p_2}{e(H)}=\frac{p_1}{e(F)}+\frac{p_2}{e(H)}\]
which completes the proof.
\end{proof}

Given a graph $F$ and non-negative integer $s$, let $s\cdot F$ be the disjoint union of $s$ copies of $F$. The following corollary is easily derived from Theorem~\ref{th:unions}. 

\begin{cor}
\label{cor:unions}
If $(H_1,H_2)$ is $(p_1,p_2)$-common, then $(s_1\cdot H_1,s_2\cdot H_2)$ is $(p_1,p_2)$-common for any positive integers $s_1$ and $s_2$. 
\end{cor}

\section{Algebraic Expansion and Odd Cycles}
\label{sec:odd}

One approach to analyzing $t(H,W)$ is to re-parameterize $W$ as a constant function plus a ``perturbation'' kernel and expand the product inside of the integral in the definition of $t(H,W)$. This is a standard trick that is used in several papers on common graphs~\cite{CsokaHubaiLovasz23,JaggerStovicekThomason96,GrzesikLeeLidickyVolec22,HancockKralKrncVolec23,Sidorenko96,Thomason97} and Sidorenko's Conjecture~\cite{Lovasz11}. Given a graph $H$ and a set $E \subseteq E(H)$, let $H[E]$ be the graph with vertex set $V(H)$ and edge set $E$.

\begin{lem}
\label{lem:expansion}
Let $H$ be a graph, let $W$ be a kernel, let $p\in\mathbb{R}$ and let $U$ be the kernel defined by $U:=W-p$. Then
\[t(H,W)=\sum_{E\subseteq E(H)}p^{e(H)-|E|}t(H[E],U).\]
\end{lem}

\begin{proof}
We have
\[t(H,W) = \int\displaylimits_{[0,1]^{V(H)}}\prod_{uv\in E(H)}W(x_u,x_v)dx_{V(H)}=\int\displaylimits_{[0,1]^{V(H)}}\prod_{uv\in E(H)}\left(p+U(x_u,x_v)\right)dx_{V(H)}.\]
Expanding the product inside the above integral yields
\[\int\displaylimits_{[0,1]^{V(H)}} \left(\sum_{E\subseteq E(H)}p^{e(H)-|E|}\prod_{uv\in E}U(x_u,x_v)\right)dx_{V(H)}=\sum_{E\subseteq E(H)}p^{e(H)-|E|}t(H[E],U).\]
\end{proof}

Lemma~\ref{lem:expansion} leads us to an equivalent condition for a pair of graphs to be $(p_1,p_2)$-common.

\begin{lem}
\label{lem:expansionLem}
Let $p_1,p_2\in (0,1)$ such that $p_1+p_2=1$ and let $H_1$ and $H_2$ be non-empty graphs. Then $(H_1,H_2)$  is $(p_1,p_2)$-common if and only if 
\begin{equation}\label{eq:expansionBound}\sum_{\substack{E\subseteq E(H_1)\\ |E|\geq2}}\frac{t(H_1[E],U)}{e(H_1)p_1^{|E|-1}} + \sum_{\substack{E\subseteq E(H_2)\\ |E|\geq2}}(-1)^{|E|}\frac{t(H_2[E],U)}{e(H_2)p_2^{|E|-1}}\end{equation}
is non-negative for any kernel $U$ such that $-p_1 \leq U(x,y)\leq p_2$ for all $x,y\in [0,1]$.
\end{lem}

\begin{proof}
Let $W_1$ and $W_2$ be graphons such that $W_1+W_2=1$ and let $U=W_1-p_1$. Then $W_1=p_1+U$ and $W_2=p_2-U$. Moreover, it holds that $-p_1 \leq U(x,y)\leq 1-p_1=p_2$ for all $x,y\in [0,1]$ since $W$ is a graphon. By Lemma~\ref{lem:expansion} and the fact that $t(F,cU) = c^{e(F)}t(F,U)$ every kernel $U$ and $c\in\mathbb{R}$, the expression
\[\frac{t(H_1,W_1)}{e(H_1)p_1^{e(H_1)-1}}+\frac{t(H_2,W_2)}{e(H_2)p_2^{e(H_2)-1}}\]
can be rewritten as
\[\sum_{E\subseteq E(H_1)}\frac{p_1^{e(H_1)-|E|}t(H_1[E],U)}{e(H_1)p_1^{e(H_1)-1}} + \sum_{E\subseteq E(H_2)}(-1)^{|E|}\frac{p_2^{e(H_2)-|E|}t(H_2[E],U)}{e(H_2)p_2^{e(H_2)-1}}\]
or, equivalently,
\[\sum_{E\subseteq E(H_1)}\frac{t(H_1[E],U)}{e(H_1)p_1^{|E|-1}} + \sum_{E\subseteq E(H_2)}(-1)^{|E|}\frac{t(H_2[E],U)}{e(H_2)p_2^{|E|-1}}.\]
The contribution of the term $E=\emptyset$ to this expression is precisely $\frac{p_1}{e(H_1)}+\frac{p_2}{e(H_2)}$. Thus,  $(H_1,H_2)$ is $(p_1,p_2)$-common if and only if the other terms sum to a non-negative value for every possible choice of $U$. The edge sets of cardinality one contribute a total of $t(K_2,U)$ to the first summation and $-t(K_2,U)$ to the second, and so their contributions cancel. This completes the proof. 
\end{proof}

Our focus in the rest of this section and the one that follows it is to use Lemma~\ref{lem:expansionLem} to obtain necessary conditions for a pair $(H_1,H_2)$ to be $(p,1-p)$-common. For this, we will use several different operations on kernels. The first can be thought of as ``scaling down'' a kernel onto a $\delta$-proportion of the vertex set. 

\begin{defn}
Given a kernel $U$ and a real number $\delta\in (0,1]$, let $U^\delta$ be the kernel such that, for $x,y\in [0,1]$,
\[U^\delta(x,y)=\begin{cases}U(x/\delta,y/\delta)& \text{if }x,y\in [0,\delta],\\ 0 &\text{otherwise}.\end{cases}\]
\end{defn}

Given a graph $F$, let $i(F)$ be the number of isolated vertices of $F$. 

\begin{lem}
\label{lem:scaleDown}
For any graph $H$, kernel $U$ and $\delta\in (0,1]$,
\[t(F,U^\delta) = \delta^{v(F)-i(F)}t(F,U).\]
\end{lem}

\begin{proof}
We may assume that $F$ has no isolated vertices. Indeed, for any isolated vertex $w$, the variable corresponding to $w$ does not appear in the product $\prod_{uv\in E(F)}U(x_u,x_v)$. So, by integrating out that variable over $[0,1]$, we see that $t(F,U)=t(F-w,U)$ for any kernel $U$, and so we can conclude by induction on $v(F)$. 

So, assume that $F$ has no isolated vertices. By definition, $U^\delta(x,y)=0$ whenever $x$ or $y$ is not in $[0,\delta]$. Therefore, if $x_{V(F)}\in [0,1]^{V(F)}$ such that there is a vertex $v$ with $x_v\notin [0,\delta]$, then the product $\prod_{uv\in E(F)}U^\delta(x_u,x_v)$ is automatically zero. Therefore,
\[t(F,U^\delta)=\int\displaylimits_{[0,\delta]^{V(F)}}\prod_{uv\in E(F)}U^\delta(x_u,x_v)dx_{V(F)} = \int\displaylimits_{[0,\delta]^{V(F)}}\prod_{uv\in E(F)}U(x_u/\delta,x_v/\delta)dx_{V(F)}\]
\[=\delta^{v(F)}t(F,U).\]
This completes the proof.
\end{proof}

In order to simplify \eqref{eq:expansionBound}, it is often useful to choose $U$ to be ``regular'' in a certain sense.

\begin{defn}
Given a kernel $U$ and $x\in [0,1]$, the \emph{degree} of $x$ is defined to be
\[d_U(x):=\int_0^1U(x,y)dy.\]
\end{defn}

\begin{defn}
Given a kernel $U$ and $d\in \mathbb{R}$, we say that $U$ is \emph{$d$-regular} if $d_U(x)=d$ for almost every $x\in [0,1]$.
\end{defn}
Observe that if $U$ is $d$-regular, then $t(K_2,U) = d$.  The next lemma allows us to reduce the calculation of $t(F,U)$ to the calculation of $t(F',U)$, whenever $U$ is $d$-regular and $F'$ is a subgraph obtained from $F$ by removing vertices of degree one. 

\begin{lem}
\label{lem:dregular}
If $U$ is a $d$-regular kernel and $F$ is a graph containing a vertex $w$ of degree one, then
\[t(F,U) = d\cdot t(F-w,U).\]
\end{lem}

\begin{proof}
Let $z$ be the unique neighbour of $w$ in $F$. We have
\[t(F,U)=\int\displaylimits_{[0,1]^{V(F)}}\prod_{uv\in E(F)}U(x_u,x_v)dx_{V(F)}\]
\[= \int\displaylimits_{[0,1]^{V(F)\setminus\{w\}}}\left(\int_0^1U(x_z,x_w)dx_w\right)\prod_{uv\in E(F)\setminus\{wz\}}U(x_u,x_v)dx_{V(F)\setminus\{w\}}=d\cdot t(F - w,U)\]
as desired.
\end{proof}

We now use the tools built up in this section to prove Theorem~\ref{th:Girth}, which we restate here for convenience.   

\repeatThm{Girth}

\begin{proof}
We prove the contrapositive. Suppose that $g(H_1)=k$ where $k$ is an odd integer and that $g(H_2)\geq k$. Note that, if $g(H_2)>k$, then $c_k(H_2)=0$; so, in any case, we can assume, without loss of generality, that
\begin{equation}\label{eq:contraCk}\frac{c_k(H_1)}{e(H_1)p_1^{k-1}} > \frac{c_k(H_2)}{e(H_2)p_2^{k-1}}.\end{equation}
Let $I_1$ and $I_2$ be intervals of length $1/2$ partitioning $[0,1]$ and let $B$ be the kernel defined by
\[B(x,y)=\begin{cases}-1&\text{if }(x,y)\in I_i^2\text{ for some }i\in\{1,2\},\\ 1&\text{otherwise}.\end{cases}\]
Observe that the homomorphism density of a cycle $C_\ell$ in $B$ is precisely $2^{-\ell}$ times the trace of the $\ell$th power of the $2 \times 2$ matrix $\left[\begin{smallmatrix} -1 & ~1 \\ ~1 & -1 \end{smallmatrix}\right]$. Since the eigenvalues of this matrix are $0$ and $-2$, we see that $t(C_\ell,B)=(-1)^\ell$.

Now, let $p=\min\{p_1,p_2\}$, let $\delta\in (0,1)$ be a small real number, to be chosen later, and define $U:=p\cdot B^\delta$. By construction, $U$ is $0$-regular. Therefore, by Lemma~\ref{lem:dregular}, we have $t(F,U)=0$ for any graph $F$ which has a vertex of degree one. In particular, since $g(H_2)\geq g(H_1)=k$, this means that $t(F,U)=0$ for any subgraph $F$ of $H_1$ or $H_2$ with fewer than $k$ edges. Also, using Lemma~\ref{lem:scaleDown}, the fact that $k$ is odd and $t(C_\ell,B)=(-1)^\ell$, we see that $$t(C_k,U)=  t(C_k, p\cdot B^\delta) = p^k t(C_k, B^\delta) = p^k\delta^k t(C_k, B) = -p^k\delta^k.$$

For $i\in \{1,2\}$, let $b_i$ be the number of subgraphs of $H_i$ with at least $k+1$ non-isolated vertices. By Lemma~\ref{lem:scaleDown}, for any graph $F$ we have $$t(F,U)\leq \delta^{v(F)-i(F)}t(F,B) \le \delta^{v(F)-i(F)},$$
as by definition of $B$, $t(F,B) \le 1$.
Putting all of this together, the expression in \eqref{eq:expansionBound} can be bounded above as follows:
\[\sum_{\substack{E\subseteq E(H_1)\\ |E|\geq2}}\frac{t(H_1[E],U)}{e(H_1)p_1^{|E|-1}} + \sum_{\substack{E\subseteq E(H_2)\\ |E|\geq2}}(-1)^{|E|}\frac{t(H_2[E],U)}{e(H_2)p_2^{|E|-1}}\]
\[\leq \frac{-c_k(H_1)p^k\delta^k}{e(H_1)p_1^{k-1}} + \frac{c_k(H_2)p^k\delta^k}{e(H_2)p_2^{k-1}}+\sum_{i=1}^2\frac{b_i\delta^{k+1}}{e(H_i)p^{e(H_i)-1}}.\]
Now, by \eqref{eq:contraCk}, if $\delta$ is chosen small enough with respect to $H_1,H_2,p_1$ and $p_2$, then the above expression is negative. This completes the proof. 
\end{proof}

From Theorem~\ref{th:Girth}, we obtain a strong constraint on $(p_1,p_2)$ in the case that $H_1$ and $H_2$ both have odd girth. 

\begin{cor}
\label{cor:specificP}
Let $H_1$ and $H_2$ be graphs of odd girth, let $k=g(H_1)$ and let $p\in(0,1)$. If $(H_1,H_2)$ is $(p,1-p)$-common, then $g(H_2)=k$ and $p=\frac{1}{\alpha+1}$ where
\[\alpha=\left(\frac{c_k(H_2)e(H_1)}{c_k(H_1)e(H_2)}\right)^{1/(k-1)}.\]
\end{cor}

\begin{proof}
Let $k=g(H_1)$. By Theorem~\ref{th:Girth}, we must also have $g(H_2)=k$. Also, we must have
\[\frac{c_k(H_1)}{e(H_1)p^{k-1}} = \frac{c_k(H_2)}{e(H_2)(1-p)^{k-1}}.\]
The result follows by solving for $p$. 
\end{proof}

\section{Tensor Products and Cliques of Order Four}
\label{sec:K4}

Our goal in this section is to prove that, if $(H_1,H_2)$ is $(p,1-p)$-common, then both of $H_1$ and $H_2$ must be $K_4$-free. For this, we will use the notion of a tensor products of kernels, which features prominently in constructions showing that certain graphs are uncommon~\cite{Thomason97,EvenZoharLinial15,CsokaHubaiLovasz23}. Fix $\psi:[0,1]\to [0,1]^2$ to be a measure preserving map, which will be used throughout the section. Note that such a $\psi$ exists as $[0,1]$ and $[0,1]^2$ are atomless standard probability spaces. For $x\in [0,1]$, let $\psi_1(x)$ and $\psi_2(x)$ be such that $\psi(x)=(\psi_1(x),\psi_2(x))$.

\begin{defn}
Let $U_1$ and $U_2$ be kernels. The \emph{tensor product} of $U_1$ and $U_2$ is the function $U_1\otimes U_2:[0,1]^2\to \mathbb{R}$ defined by
\[U_1\otimes U_2(x,y)=U_1(\psi_1(x),\psi_1(y))\cdot U_2(\psi_2(x),\psi_2(y))\]
for all $(x,y)\in[0,1]^2$. 
\end{defn}

The following (standard) lemma describes a key property of the tensor product.

\begin{lem}[E.g.~{\cite[Equation~(7.17)]{Lovasz12}}]
\label{lem:tensorLem}
If $F$ is a graph and $U_1$ and $U_2$ are kernels, then
\[t(F,U_1\otimes U_2)=t(F,U_1)\cdot t(F,U_2).\]
\end{lem}

In order to prove Theorem~\ref{th:KFour}, we modify a construction of Jagger, \v{S}\v{t}ov\'i\v{c}ek and Thomason~\cite{JaggerStovicekThomason96} to construct a kernel $U$ such that the expression in \eqref{eq:expansionBound} is negative. The constructed kernel will be obtained by taking a tensor power of the following kernel and scaling it down to a small proportion of the vertex set. 

\begin{defn}\label{defn:kernel}
Let $I_1,I_2,I_3,I_4$ be a partition of $[0,1]$ into four intervals of length $1/4$. Let $K$ be the kernel defined by
\[K(x,y):=\begin{cases}1 & \text{if }(x,y)\in I_i^2\text{ for some }i\in\{1,2,3,4\},\\ -1 & \text{otherwise}.\end{cases}\]
\end{defn}

The next lemma can be seen as a weak form of Lemma~10 from~\cite{JaggerStovicekThomason96}; we have distilled it down to contain only the ingredients that we will need here. We note that the proof of this lemma is fairly involved and relies on Fourier analytic ideas; for this reason, we will not repeat the full analysis here. An alternative argument showing that $K_4$ is uncommon using the same kernel $K$, but without Fourier analysis, was recently provided by Cs\'oka, Hubai and Lov\'asz~\cite{CsokaHubaiLovasz23}. 

\begin{lem}[Jagger, \v{S}\v{t}ov\'i\v{c}ek and Thomason~\cite{JaggerStovicekThomason96}]
\label{lem:blackBox}
If $K$ is the kernel in Definition~\ref{defn:kernel}, then
\begin{enumerate}
\stepcounter{equation}
    \item\label{eq:K4one} $t(K_4,K)=-1/2$,
\stepcounter{equation}
    \item\label{eq:larger} $|t(F,K)|\leq 1/2$ for every graph $F$ with at least one edge, and
\stepcounter{equation}
    \item\label{eq:others} $|t(F,K)|\leq 1/4$ if $F$ is a connected graph on three or four vertices which is not isomorphic to $K_4$. 
\end{enumerate}
\end{lem}

We use Lemma~\ref{lem:blackBox} to prove Theorem~\ref{th:KFour} (restated below). Given a kernel $U$, let $U^{\otimes 0}:=1$ and, for $k\geq1$, let $U^{\otimes k}:=U^{\otimes (k-1)}\otimes U$. 

 \repeatThm{KFour}

\begin{proof}
Our goal is to apply Lemma~\ref{lem:expansionLem}. To that end, we construct a kernel $U$ such that the expression in \eqref{eq:expansionBound} is negative. 

Let $p:=\min\{p_1,p_2\}$. Let $k$ be an integer and let $\delta$ be a real number, both of which will be chosen later. We let $U=p\cdot\left(K^{\otimes (2k+1)}\right)^\delta$, where $K$ is defined in Definition~\ref{defn:kernel}. Then, by Lemmas~\ref{lem:scaleDown} and~\ref{lem:tensorLem}, we have
\[t(F,U) = p^{e(F)}\delta^{v(F)-i(F)}t(F,K)^{2k+1}\]
for any graph $F$. In particular, $t(K_4,U) = p^{6}\delta^4(-1/2)^{2k+1}$ by \eqref{eq:K4one}. Given a graph $F$, let $F^*$ be the graph obtained from $F$ by deleting all isolated vertices. Note that $K_2\sqcup K_2$ is the only disconnected graph on at most $4$ vertices without isolated vertices and that $t(K_2\sqcup K_2,K)=t(K_2,K)^2\leq 1/4$ by \eqref{eq:larger}. So, by Lemma~\ref{lem:blackBox}, we get
\[|t(F,U)|\leq \begin{cases}\delta^5(1/2)^{2k+1} & \text{if }v(F)-i(F)\geq 5,\\ 
(1/4)^{2k+1} & \text{if }3\leq v(F)-i(F)\leq 4\text{ and }F^*\notin\{K_2,K_4\}.\end{cases}\]

Now, for $i\in\{1,2\}$, let $b_{i}$ be the number of non-empty subsets $E$ of $E(H_i)$ such that $H_i[E]^*$ has at least five vertices. Let $s_{i}$ be the number of non-empty subsets $E$ of $E(H_i)$ such that $H_i[E]^*$ has at most four vertices and is neither isomorphic to $K_2$ nor to $K_4$. Also, let $n_i$ be the number of non-empty subsets $E$ of $E(H_i)$ such that $H_i[E]^*$ is isomorphic to $K_4$. The expression in \eqref{eq:expansionBound} can be bounded as follows:
\[\sum_{\substack{E\subseteq E(H_1)\\ |E|\geq2}}\frac{t(H_1[E],U)}{e(H_1)p_1^{|E|-1}} + \sum_{\substack{E\subseteq E(H_2)\\ |E|\geq2}}(-1)^{|E|}\frac{t(H_2[E],U)}{e(H_2)p_2^{|E|-1}}\]
\[\leq \sum_{i=1}^2\left(\frac{b_i\delta^5(1/2)^{2k+1}}{e(H_i)p_i^{e(H_i)-1}} + \frac{s_i(1/4)^{2k+1}}{p_i^5e(H_i)} + \frac{n_i\delta^4(-1/2)^{2k+1}p^6}{p_i^6e(H_i)}\right).\]
To complete the proof, first choose $\delta$ small enough with respect to $H_1,H_2,p_1$ and $p_2$ so that
\[\sum_{i=1}^2\frac{b_i\delta^5}{e(H_i)p_i^{e(H_i)-1}}< \frac{1}{2}\sum_{i=1}^2\frac{n_i\delta^4p^6}{p_i^6e(H_i)}.\]
Now, let $k$ be chosen large enough with respect to $H_1,H_2,p_1,p_2$ and $\delta$ so that
\[\sum_{i=1}^2\frac{s_i(1/4)^{2k+1}}{p_i^5e(H_i)} < \frac{1}{2}\sum_{i=1}^2\frac{n_i\delta^4(1/2)^{2k+1}p^6}{p_i^6e(H_i)}.\]
Putting this together, we get that the expression in \eqref{eq:expansionBound} is negative, and so $(H_1,H_2)$ is not $(p_1,p_2)$-common. 
\end{proof}

\begin{rem}
\label{rem:local}
In~\cite{CsokaHubaiLovasz23}, it is shown that any graph $H$ containing $K_4$ fails to be common in a certain ``local'' sense. Specifically, this corresponds to showing, in the case $p_1=p_2=1/2$ and $H_1=H_2=H$, that, for every $\varepsilon>0$, there exists a kernel $U$ with $\|U\|_\infty\leq \varepsilon$ such that the expression in \eqref{eq:expansionBound} is negative. The construction in the previous lemma easily yields a similar statement in the asymmetric setting. The only change is to start by letting $U=\min\{\varepsilon,p_1,p_2\}\cdot\left(K^{\otimes(2k+1)}\right)^\delta$ and then choosing $\delta$ and $k$ to additionally depend on $\varepsilon$ at the end of the proof. 
\end{rem}

\section{The Square and the Pentagon}
\label{sec:flags}

Next, we investigate the values of $p\in(0,1)$ such that the pair $(C_4,C_5)$ is $(p,1-p)$-common. Specifically, we prove Theorems~\ref{th:Flag} and \ref{th:CFourCFiveNot}, which we restate here for convenience.

\repeatThm{Flag}

\repeatThm{CFourCFiveNot}

The proof of Theorem~\ref{th:CFourCFiveNot} is elementary, and so we can present it now. 

\begin{proof}[Proof of Theorem~\ref{th:CFourCFiveNot}]
First, we observe that, if $p>0.518$, then
\begin{equation}
\label{eq:pIneq}
40p^4 + 32(1-p)p^3-5>0.
\end{equation}
Our goal is to show that, if \eqref{eq:pIneq} holds, then $(C_4,C_5)$ is not $(p,1-p)$-common. Let $W_1$ be the graphon defined by
\[W_1(x,y)=\begin{cases}1&\text{if }x,y\in[0,1/2] \text{ or }x,y\in(1/2,1],\\ 0&\text{otherwise}\end{cases}\]
and define $W_2=1-W_1$. Then $t(C_4,W_1)=1/8$ and $t(C_5,W_2)=0$. Thus, $(C_4,C_5)$ fails to be $(p,1-p)$-common for any $p$ such that
\[\frac{(1/8)}{4p^3}+\frac{0}{5(1-p)^4}<\frac{p}{4}+\frac{1-p}{5}\]
After a bit of rearranging, one can show that this inequality is equivalent to \eqref{eq:pIneq}. The result follows. 
\end{proof}

The proof of Theorem~\ref{th:Flag} was discovered and verified via the flag algebra method of Razborov~\cite{Razborov07}. While the ideas underpinning this method are beautiful and natural, it is unfortunately rather technical to introduce the method in full generality. For this reason, we will aim for an intuitive introduction of only those aspects of the method that we need. 

\begin{defn}
Given a graph $J$ and graphon $W$, define the \emph{induced homomorphism density} of $J$ in $W$ to be
\[t_{\ind}(J,W):=\int\limits_{[0,1]^{V(J)}}\prod_{uv\in E(J)}W(x_u,x_v)\prod_{uv\in E(\overline{J})}(1-W(x_u,x_v))dx_{V(J)}.\]
\end{defn}

Let $L_1,\dots,L_{1024}$ be the $1024=2^{\binom{5}{2}}$ labelled graphs with vertex set $[5]$, where $[n]:=\{1,2,\dots,n\}$ for all $n\geq1$. We pause for a simple observation.

\begin{obs}
\label{obs:sumToOne}
For any graphon $W$, we have $\sum_{m=1}^{1024}t_{\ind}(L_m,W)=1$.
\end{obs}

\begin{proof}
Observe that
\[1=\int_{[0,1]^5}dx_1\cdots dx_5=\int_{[0,1]^5}\prod_{1\leq i\neq j\leq 5}(W(x_i,x_j) + (1-W(x_i,x_j)))dx_1\cdots dx_5.\]
Expanding the product in this integral yields the sum of $t_{\ind}(L,W)$ over all distinct graphs $L$ with vertex set $[5]$, as desired.
\end{proof}

Clearly, $t_{\ind}(J,W)=t_{\ind}(L,W)$ if $J$ and $L$ are isomorphic, and so the equation in Observation~\ref{obs:sumToOne} contains many redundant terms. To remedy this, let $J_1,\dots,J_{34}$ denote the 34 labelled graphs with vertex set $[5]$, up to isomorphism. For any graph $J$, an \emph{automorphism} of $J$ is an isomorphism from $J$ to itself. Let $\Aut(J)$ denote the set of all automorphisms of $J$ and $\aut(J):=|\Aut(J)|$. For $1\leq s\leq 34$, the number of indices $1\leq m\leq 1024$ such that $J_s$ is isomorphic to $L_m$ is precisely the number of cosets of $\Aut(J_s)$ viewed as a subgroup of the symmetric group acting on $[5]$, which is $\frac{5!}{\aut(J_s)}=\frac{120}{\aut(J_s)}$. Thus, letting $d(J_s,W):=\frac{120}{\aut(J_s)}\cdot t_{\ind}(J_s,W)$ for $1\leq s\leq 34$, we see that the equation in Observation~\ref{obs:sumToOne} can be rewritten as
\begin{equation}\label{eq:sumToOne}\sum_{s=1}^{34}d(J_s,W)=1.\end{equation}
We call $d(J_s,W)$ the \emph{induced density} of $J_s$ in $W$. 

Next, we establish a connection between $t(H,W)$ and the quantities $d(J_s,W)$ for $1\leq s\leq 34$ for any graph $H$ on at most 5 vertices. A \emph{homomorphism} from a graph $H$ to a graph $J$ is a function $\varphi:V(H)\to V(J)$ such that $\varphi(u)\varphi(v)\in E(J)$ whenever $uv\in E(H)$. Let $\hom_{\inj}(H,J)$ denote the number of injective homomorphisms from $H$ to $J$. The \emph{injective homomorphism density} of $H$ in $J$ is defined to be
\[t_{\inj}(H,J):= \frac{\hom_{\inj}(H,J)(v(J)-v(H))!}{v(J)!}.\]
In other words, $t_{\inj}(H,J)$ is the probability that a random injective function from $V(H)$ to $V(J)$ is a homomorphism. 

\begin{lem}
\label{lem:tinj}
For any graphon $W$ and graph $H$ on at most 5 vertices,
\[t(H,W)=\sum_{s=1}^{34}t_{\inj}(H,J_s)d(J_s,W).\]
\end{lem}

\begin{proof}
For this proof, it is convenient to define, for a graphon $W$, the notion of a $W$-random graph. For $n\geq1$, a \emph{$W$-random graph} of order $n$ is the graph with vertex set $[n]:=\{1,\dots,n\}$ obtained by selecting $x_1,\dots,x_n$ from $[0,1]$ uniformly at random and independently of one another and adding an edge from $i\in[n]$ to $j\in[n]$ with probability $W(x_i,x_j)$. For each $1\leq m\leq 1024$, we can interpret $t_{\ind}(L_m,W)$ as being the probability that a $W$-random graph of order 5 is equal to $L_m$ (with the same vertex labelling). Thus, for $1\leq s\leq 34$, the quantity $d(J_s,W)$ is the probability that a $W$-random graph of order 5 is isomorphic to $J_s$. 

Now, for a graph $H$ with at most $5$-vertices, consider the following procedure. First, sample a $W$-random graph $J$ of order $5$. Next, choose an injective function $\varphi$ from $V(H)$ to $[5]$ uniformly at random. The probability that $\varphi$ is a homomorphism is precisely the sum over all $1\leq s\leq 34$ of the probability that $J$ is isomorphic to $J_s$ times the conditional probability that $\varphi$ is a homomorphism from $H$ to $J$ given that $J$ is isomorphic to $J_s$, which is just the injective homomorphism density of $H$ in $J_s$. That is,
\[\mathbb{P}(\varphi\text{ is a homomorphism from $H$ to $J$})=\sum_{s=1}^{34}t_{\inj}(H,J_s)d(J_s,W).\]
On the other hand, the probability that $\varphi$ is a homomorphism from $H$ to $J$ is precisely equal to the probability that each of the edges $\varphi(u)\varphi(v)$ for $uv\in E(H)$ are present in $J$. Since $\varphi$ is injective, the tuple $(x_{\varphi(u)}:u\in V(H))$ is just a uniformly random element of $[0,1]^5$. Thus, since $J$ is a $W$-random graph, we have
\[\mathbb{P}(\varphi\text{ is a homomorphism from $H$ to $J$})=\int\limits_{[0,1]^{V(H)}}\prod_{uv\in E(H)}W(x_u,x_v)d_{V(H)}=t(H,W).\]
This completes the proof. 
\end{proof}

Next, we introduce a multiplication operation. Consider the labelled graphs $F_{i,a}$ for $1\leq i\leq 4$ and $1\leq a\leq 8$ depicted below; these graphs are referred to as \emph{flags}. In each of these depictions, the square vertices are labelled $1,2$ and $3$ in increasing order from left to right and the round vertex is labelled $4$. Edges are depicted by solid lines and non-edges by dotted lines. The graphs $F_{i,1},\dots,F_{i,8}$ have precisely $i-1$ edges among the vertices $1,2$ and $3$ and are depicted on the $i$th row of the diagram. For $1\leq i\leq 4$, we let $F_i$ be the subgraph of $F_{i,1}$ induced by $\{1,2,3\}$. 
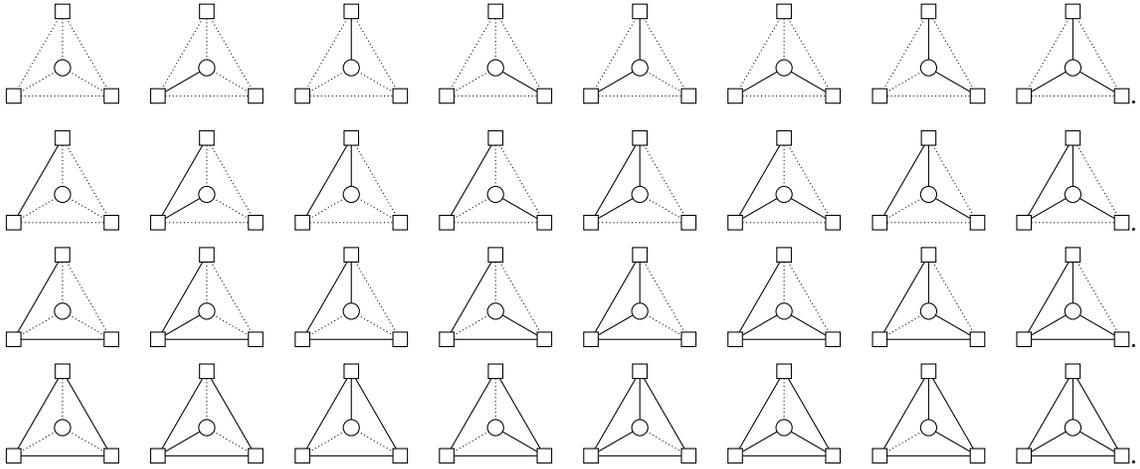
\[\ZeroZero \quad \ZeroOne \quad \ZeroTwo \quad \ZeroThree \quad \ZeroFour \quad \ZeroFive \quad \ZeroSix \quad \ZeroSeven.\]
\[\OneZeroZero \quad \OneZeroOne \quad \OneZeroTwo \quad \OneZeroThree \quad \OneZeroFour \quad \OneZeroFive \quad \OneZeroSix \quad \OneZeroSeven.\]
\[\TwoZeroZero \quad \TwoZeroOne \quad \TwoZeroTwo \quad \TwoZeroThree \quad \TwoZeroFour \quad \TwoZeroFive \quad \TwoZeroSix \quad \TwoZeroSeven.\]
\[\ThreeZeroZero \quad \ThreeZeroOne \quad \ThreeZeroTwo \quad \ThreeZeroThree \quad \ThreeZeroFour \quad \ThreeZeroFive \quad \ThreeZeroSix \quad \ThreeZeroSeven.\]
For $1\leq i\leq 4$ and a graphon $W$, let $t_{\ind,3}(F_{i},W):[0,1]^3\to \mathbb{R}$ be defined by
\[t_{\ind,3}(F_{i},W)(x_1,x_2,x_3):=\prod_{uv\in E(F_{i})}W(x_u,x_v)\prod_{uv\in E(\overline{F_{i}})}(1-W(x_u,x_v)).\]
Also, for $1\leq i\leq 4$, $1\leq a\leq 8$ and a graphon $W$, let $t_{\ind,3}(F_{i,a},W):[0,1]^3\to \mathbb{R}$ be such that
\[t_{\ind,3}(F_{i,a},W)(x_1,x_2,x_3):=\int_0^1\prod_{uv\in E(F_{i,a})}W(x_u,x_v)\prod_{uv\in E(\overline{F_{i,a}})}(1-W(x_u,x_v))dx_4.\]
Note that
\[\int_0^1\int_0^1\int_0^1t_{\ind,3}(F_{i,a},W)(x_1,x_2,x_3)dx_1dx_2dx_3=t_{\ind}(F_{i,a},W).\]
For $1\leq i\leq4$ and $1\leq a,b\leq 8$, define $t_{\ind,3}(F_{i,a}\cdot F_{i,b},W):[0,1]^3\to \mathbb{R}$ to be the function such that $t_{\ind,3}(F_{i,a}\cdot F_{i,b},W)(x_1,x_2,x_3)$ is equal to
\[\begin{cases}\displaystyle\frac{t_{\ind,3}(F_{i,a},W)(x_1,x_2,x_3)\cdot t_{\ind,3}(F_{i,b},W)(x_1,x_2,x_3)}{t_{\ind,3}(F_i,W)(x_1,x_2,x_3)}  &\text{if }t_{\ind,3}(F_i,W)(x_1,x_2,x_3)\neq0,\\
0 &\text{otherwise}.\end{cases}\]
The next lemma is key to our application of the flag algebra method. Given an $n\times n$ matrix $M$ and $1\leq a,b\leq n$, let $M(a,b)$ denote the element of $M$ in the $a$th row and $b$th column. A real $n\times n$ symmetric matrix $M$ is \emph{positive semidefinite}, written $M\succcurlyeq 0$, if all of its eigenvalues are non-negative. Alternatively, it is positive semidefinite if $z^TMz\geq0$ for every real column vector $z$ of length $n$. 

\begin{lem}
\label{lem:PSD}
For $1\leq i\leq4$, let $M_i$ be a $8\times 8$ positive semidefinite matrix. Then, for any graphon $W$,
\[\sum_{i=1}^4\sum_{a=1}^8\sum_{b=1}^8\int_0^1\int_0^1\int_0^1M_i(a,b)\cdot t_{\ind,3}(F_{i,a}\cdot F_{i,b},W)(x_1,x_2,x_3)dx_1dx_2dx_3\geq0.\]
\end{lem}

\begin{proof}
We prove the stronger statement that
\[\sum_{a=1}^8\sum_{b=1}^8M_i(a,b)\cdot t_{\ind,3}(F_{i,a}\cdot F_{i,b},W)(x_1,x_2,x_3)\geq0\]
for all $1\leq i\leq 4$ and $x_1,x_2,x_3\in [0,1]$. First, if $t_{\ind,3}(F_i,W)=0$, then every term of the sum is zero and we are done. Otherwise, we get
\[\sum_{a=1}^8\sum_{b=1}^8M_i(a,b)\cdot t_{\ind,3}(F_{i,a}\cdot F_{i,b},W)(x_1,x_2,x_3)\]
\[= \frac{\sum_{a=1}^8\sum_{b=1}^8M_i(a,b)\cdot t_{\ind,3}(F_{i,a},W)(x_1,x_2,x_3)\cdot t_{\ind,3}(F_{i,b},W)(x_1,x_2,x_3)}{t_{\ind,3}(F_i,W)(x_1,x_2,x_3)}\]
which is non-negative because $M_i$ is positive semidefinite.
\end{proof}

In describing the way that we apply the flag algebra method, the final piece of the puzzle is to relate the integrals of the functions $t(F_{i,a}\cdot F_{i,b},W)(x_1,x_2,x_3)$ to the densities $d(J_s,W)$ for $1\leq s\leq 34$. This boils down to a simple double-counting argument akin to the proof of the Handshaking Lemma. Let us introduce this argument with an example. Consider the integral
\[\int_0^1\int_0^1\int_0^1 t_{\ind,3}(F_{1,1}\cdot F_{1,1},W)(x_1,x_2,x_3)dx_1dx_2dx_3.\]
Recall that $F_{1,1}$ is the edgeless graph with vertex set $[4]$. After a change of variables, we see that the above integral can be rewritten as
\[\int_{[0,1]^5}\prod_{1\leq i<j\leq 3}(1-W(x_i,x_j)) \prod_{i=1}^3(1-W(x_i,x_4))\prod_{i=1}^3(1-W(x_i,x_5))dx_1\cdots dx_5.\]
Multiplying the integrand by $1=(W(x_4,x_5) + (1-W(x_4,x_5)))$ and expanding yields the sum of the induced homomorphism densities of the $5$-vertex graph with zero edges and the $5$-vertex graph with one edge; call these graphs $J_1$ and $J_2$, respectively. Recalling that $t_{\ind}(J,W)=\frac{\aut(J)}{120}\cdot d(J,W)$ for any graph $J$ on five vertices, we see that 
\[\int_0^1\int_0^1\int_0^1 t_{\ind,3}(F_{1,1}\cdot F_{1,1},W)(x_1,x_2,x_3)dx_1dx_2dx_3=d(J_1,W) + \frac{12}{120}\cdot d(J_2,W).\]
Similarly, for any $1\leq i\leq 4$ and $1\leq a,b\leq 8$, the integral of $t_{\ind,3}(F_{i,a}\cdot F_{i,b},W)(x_1,x_2,x_3)$ over $[0,1]^3$ can be expressed as $\frac{\aut(J_p)}{120}\cdot d(J_p,W) + \frac{\aut(J_q)}{120}\cdot d(J_q,W)$ for two particular $5$-vertex graphs $J_p$ and $J_q$ constructed as follows. Start with the disjoint union of $F_{i,a}$ and $F_{i,b}$. Then, for each $1\leq j\leq 3$, identify the vertex labelled $j$ in $F_{i,a}$ with the vertex labelled $j$ in $F_{i,b}$. Let $J_p$ be the $5$-vertex graph obtained from this identification, and $J_q$ be the graph obtained from $J_p$ by adding an edge between the vertices labelled $4$ in $F_{i,a}$ and $F_{i,b}$. For each graph $J$ on $5$ vertices, if we let $c(F_{i,a},F_{i,b},J)$ be equal to $\frac{\aut(J)}{120}$ if $J\in\{J_p,J_q\}$ and zero otherwise, then we get the following
\begin{equation}
\label{eq:doubleCount}
\int_0^1\int_0^1\int_0^1 t_{\ind,3}(F_{i,a}\cdot F_{i,b},W)(x_1,x_2,x_3)dx_1dx_2dx_3 = \sum_{s=1}^{34}c(F_{i,a},F_{i,b},J_s)\cdot d(J_s,W).
\end{equation}
The next lemma illustrates the way in which the ideas presented in this section can be used to show that a pair $(H_1,H_2)$ of graphs on at most 5 vertices is $(p_1,p_2)$-common. This, of course, can be generalized to different sets of flags and larger graphs, but we will not need such a general statement in this paper. 

\begin{lem}
\label{lem:flagLem}
Let $H_1$ and $H_2$ be graphs on at most 5 vertices and $p_1,p_2 \in (0, 1)$ such that $p_1 + p_2 = 1$. If there exist $8\times 8$ positive semidefinite matrices $M_1,M_2,M_3$ and $M_4$ such that
\[\frac{t_{\inj}(H_1,J)}{e(H_1)p_1^{e(H_1)-1}} + \frac{t_{\inj}(H_2,\overline{J})}{e(H_2)p_2^{e(H_2)-1}} - \sum_{i=1}^4\sum_{a=1}^8\sum_{b=1}^8M_i(a,b)c(F_{i,a},F_{i,b},J)\geq\frac{p_1}{e(H_1)}+\frac{p_2}{e(H_2)}\]
for every graph $J$ on $5$ vertices, then $(H_1, H_2)$ is $(p_1, p_2)$-common.
\end{lem}

\begin{proof}
Let $W$ be any graphon. By Lemma~\ref{lem:tinj},
\[t(H_1,W) = \sum_{i=1}^{34}t_{\inj}(H_1,J_s)d(J_s,W).\]
Note that $d(J,W)=d(\overline{J},1-W)$ for any graph $J$ and so, by another application of Lemma~\ref{lem:tinj},
\[t(H_2,1-W) = \sum_{i=1}^{34}t_{\inj}(H_2,\overline{J_s})d(\overline{J_s},1-W) = \sum_{i=1}^{34}t_{\inj}(H_2,\overline{J_s})d(J_s,W).\]
Combining these two equalities, we get that
\[\frac{t(H_1,W)}{e(H_1)p_1^{e(H_1)-1}} + \frac{t(H_2,1-W)}{e(H_2)p_2^{e(H_2)-1}} = \sum_{s=1}^{34}\left(\frac{t_{\inj}(H_1,J_s)}{e(H_1)p_1^{e(H_1)-1}} + \frac{t_{\inj}(H_2,\overline{J_s})}{e(H_2)p_2^{e(H_2)-1}}\right)d(J_s,W).\]
By combining Lemma~\ref{lem:PSD} and \eqref{eq:doubleCount}, we see that the above expression is bounded below by
\[\sum_{s=1}^{34}\left(\frac{t_{\inj}(H_1,J_s)}{e(H_1)p_1^{e(H_1)-1}} + \frac{t_{\inj}(H_2,\overline{J_s})}{e(H_2)p_2^{e(H_2)-1}} - \sum_{i=1}^4\sum_{a=1}^8\sum_{b=1}^8M_i(a,b)c(F_{i,a},F_{i,b},J_s)\right)d(J_s,W).\]
Finally, by \eqref{eq:sumToOne}, this is at least
\[\min_{1\leq s\leq 34}\left(\frac{t_{\inj}(H_1,J_s)}{e(H_1)p_1^{e(H_1)-1}} + \frac{t_{\inj}(H_2,\overline{J_s})}{e(H_2)p_2^{e(H_2)-1}} - \sum_{i=1}^4\sum_{a=1}^8\sum_{b=1}^8M_i(a,b)c(F_{i,a},F_{i,b},J_s)\right)\]
which is at least $\frac{p_1}{e(H_1)}+\frac{p_2}{e(H_2)}$ by the hypothesis of the lemma. This completes the proof. 
\end{proof}

Finally, we use this lemma to prove Theorem~\ref{th:Flag}. 

\begin{proof}[Proof of Theorem~\ref{th:Flag}]
First, let us show that $(C_4,C_5)$ is $(1/2,1/2)$-common. We define the following four matrices:
\[M_1:=\frac{1}{8640}\begin{bmatrix}
25704 & 6806 & 6806 & 6806 & -8112 & -8112 & -8112 & -21786 \\ 
6806 & 10512 & -1080 & -1080 & 481 & 481 & -9334 & -6786 \\ 
6806 & -1080 & 10512 & -1080 & 481 & -9334 & 481 & -6786 \\ 
6806 & -1080 & -1080 & 10512 & -9334 & 481 & 481 & -6786 \\ 
-8112 & 481 & 481 & -9334 & 11304 & -720 & -720 & 6620 \\ 
-8112 & 481 & -9334 & 481 & -720 & 11304 & -720 & 6620 \\ 
-8112 & -9334 & 481 & 481 & -720 & -720 & 11304 & 6620 \\ 
-21786 & -6786 & -6786 & -6786 & 6620 & 6620 & 6620 & 22284 \end{bmatrix}\]
\[M_2:=\frac{1}{8640}\begin{bmatrix}
52260 & 4029 & 4029 & 31650 & -33264 & -7632 & -7632 & -43440 \\ 
4029 & 19224 & -5400 & -864 & 864 & -5469 & -6120 & -6264 \\ 
4029 & -5400 & 19224 & -864 & 864 & -6120 & -5469 & -6264 \\ 
31650 & -864 & -864 & 35976 & -35484 & -111 & -111 & -30192 \\ 
-33264 & 864 & 864 & -35484 & 37140 & 216 & 216 & 29448 \\ 
-7632 & -5469 & -6120 & -111 & 216 & 14586 & -1230 & 5760 \\ 
-7632 & -6120 & -5469 & -111 &  216 & -1230 & 14586 & 5760 \\ 
-43440 & -6264 & -6264 & -30192 & 29448 & 5760 & 5760 & 45192\end{bmatrix}\]
\[M_3:=\frac{1}{8640}\begin{bmatrix}
40104 & 7458 & 11880 & 11880 & -10095 & -10095 & -17496 & -33636 \\ 
7458 & 33852 & -3888 & -3888 & 978 & 978 & -15768 & -19722 \\
11880 & -3888 & 16560 & 3390 & -3660 & -12075 & -2463 & -9744 \\ 
11880 & -3888 & 3390 & 16560 & -12075 & -3660 & -2463 & -9744 \\ 
-10095 & 978 & -3660 & -12075 & 15366 & -144 & 105 & 9525 \\ 
-10095 & 978 & -12075 & -3660 & -144 & 15366 & 105 & 9525 \\
-17496 & -15768 & -2463 & -2463 & 105 & 105 & 24372 & 13608 \\   
-33636 & -19722 & -9744 & -9744 & 9525 & 9525 & 13608 & 40188\end{bmatrix}\]
\[M_4:=\frac{1}{8640}\begin{bmatrix}
15552 & 4912 & 4912 & 4912 & -5166 & -5166 & -5166 & -14790 \\ 
4912 & 10504 & -1728 & -1728 & 824 & 824 & -8640 & -4968 \\ 
4912 & -1728 & 10504 & -1728 & 824 & -8640 & 824 & -4968 \\ 
4912 & -1728 & -1728 & 10504 & -8640 & 824 & 824 & -4968 \\ 
-5166 & 824 & 824 & -8640 & 10760 & -1694 & -1694 & 4786 \\ 
-5166 & 824 & -8640 & 824 & -1694 & 10760 & -1694 & 4786 \\ 
-5166 & -8640 & 824 & 824 & -1694 & -1694 & 10760 & 4786 \\ 
-14790 & -4968 & -4968 & -4968 & 4786 & 4786 & 4786 & 15336\end{bmatrix}\]
All four of these matrices are positive semidefinite. To complete the proof, all that is required is to show that the inequality in Lemma~\ref{lem:flagLem} is satisfied for every graph $J$ on $5$ vertices. In fact, this inequality turns out to be an exact equality for every such graph. It would be incredibly tedious to check this inequality for every such graph by hand. For the purposes of illustration, let us just verify two simple examples; the rest of these calculations have been included in Appendix~A in an ancillary file submitted with the arxiv version of this paper:  \url{https://arxiv.org/src/2208.02045/anc/commonExtraAppendices.pdf}. Suppose that $J=K_5$. Then $t_{\inj}(C_4,K_5)=1$ and $t_{\inj}(C_5,\overline{K_5})=0$. One can check that $c(F_{4,8},F_{4,8},K_5)=1$ and that $c(F_{i,a},F_{i,b},K_5)=0$ for all other choices of $i,a$ and $b$. Therefore, in the case $J=K_5$, the left side of the inequality in Lemma~\ref{lem:flagLem} is
\[\frac{1}{4\cdot (1/2)^3} + \frac{0}{5\cdot (1/2)^4} - \frac{15336}{8640}= \frac{9}{40}=\frac{1/2}{4}+\frac{1/2}{5}\]
as desired. For a slightly more involved example, suppose that $J=K_3\sqcup K_2$. Then $t_{\inj}(C_4,J)=t_{\inj}(C_5,\overline{J})=0$. We have 
\[c(F_{2,4},F_{2,4},J)=\frac{12}{120},\]
\[c(F_{2,4},F_{2,5},J)=c(F_{2,5},F_{2,4},J)=\frac{12}{120},\]
\[c(F_{4,1},F_{4,1},J)=\frac{12}{120},\]
and $c(F_{i,a},F_{i,b},K_5)=0$ for all other choices of $i,a$ and $b$.
Therefore, the left side of the inequality in Lemma~\ref{lem:flagLem} is
\[0-\frac{35976\cdot 12}{8640\cdot 120}+\frac{35484\cdot 12}{8640\cdot 120}+\frac{35484\cdot 12}{8640\cdot 120}-\frac{15552\cdot 12}{8640\cdot 120}=\frac{9}{40}=\frac{1/2}{4}+\frac{1/2}{5}\]
as desired.

The proof for $(1/3,2/3)$ is similar, except that the matrices $M_1,M_2,M_3,M_4$ are replaced with $\frac{1}{53084160}$ multiplied by the following four matrices, respectively:

\[\fontsize{9}{10.8}\selectfont\begin{bmatrix}
42246144 & 582094 & 582094 & 582094 & -42135552 & -42135552 & -42135552 & -92141034 \\ 
582094 & 49243392 & -15842304 & -15842304 & 7796736 & 7796736 & -46514496 & -13049856 \\ 
582094 & -15842304 & 49243392 & -15842304 & 7796736 & -46514496 & 7796736 & -13049856 \\ 
582094 & -15842304 & -15842304 & 49243392 & -46514496 & 7796736 & 7796736 & -13049856 \\ 
-42135552 & 7796736 & 7796736 & -46514496 & 138599424 & 21676032 & 21676032 & 96865624 \\ 
-42135552 & 7796736 & -46514496 & 7796736 & 21676032 & 138599424 & 21676032 & 96865624 \\ 
-42135552 & -46514496 & 7796736 & 7796736 & 21676032 & 21676032 & 138599424 & 96865624 \\ 
-92141034 & -13049856 & -13049856 & -13049856 & 96865624 & 96865624 & 96865624 & 312532992\end{bmatrix}\]
\[\fontsize{9}{10.8}\selectfont\begin{bmatrix}
107984172 & -18972513 & -18972513 & 41642346 &-163858752 & -76966047 & -76966047 & -243081216 \\ 
-18972513 & 98249718 & -4893696 & -36052992 & 56408379 & -75589632 & -25408512 & 11747760 \\ 
-18972513 & -4893696 & 98249718 & -36052992 & 56408379 & -25408512 & -75589632 & 11747760 \\ 
41642346 & -36052992 & -36052992 & 162533628 & -171260478 & -38375424 & -38375424 & -198826518 \\ 
-163858752 & 56408379 & 56408379 & -171260478 & 487876608 & 83816640 & 83816640 & 233625600 \\ 
-76966047 & -75589632 & -25408512 & -38375424 & 83816640 & 359826702 & 20736000 & 244463616 \\ 
-76966047 & -25408512 & -75589632 & -38375424 & 83816640 & 20736000 & 359826702 & 244463616 \\ 
-243081216 & 11747760 & 11747760 & -198826518 & 233625600 & 244463616 & 244463616 & 1200867840\end{bmatrix}\]
\[\fontsize{9}{10.8}\selectfont\begin{bmatrix}
101172144 & 5282292 & -6580224 & -6580224 & -85681152 & -85681152 & -110108160 & -214922952 \\ 
5282292 & 126774228 & -8879739 & -8879739 & -55406592 & -55406592 & -39429504 & -177831936 \\ 
-6580224 & -8879739 & 151234560 & -62871552 & -3815424 & -96242688 & -30360144 & -4455024 \\ 
-6580224 & -8879739 & -62871552 & 151234560 & -96242688 & -3815424 & -30360144 & -4455024 \\ 
-85681152 & -55406592 & -3815424 & -96242688 & 466255872 & -44568576 & 109264896 & 245402496 \\ 
-85681152 & -55406592 & -96242688 & -3815424 & -44568576 & 466255872 & 109264896 & 245402496 \\ 
-110108160 & -39429504 & -30360144 & -30360144 & 109264896 & 109264896 & 367946496 & 108512640 \\ 
-214922952 & -177831936 & -4455024 & -4455024 & 245402496 & 245402496 & 108512640 & 1267716096\end{bmatrix}\]
\[\fontsize{9}{10.8}\selectfont\begin{bmatrix}
64971648 & -8197220 & -8197220 & -8197220 & -54645562 & -54645562 & -54645562 & -93533184 \\ 
-8197220 & 104755556 & -34466080 & -34466080 & 15801488 & 15801488 & -69672960 & -1575936 \\ 
-8197220 & -34466080 & 104755556 & -34466080 & 15801488 & -69672960 & 15801488 & -1575936 \\ 
-8197220 & -34466080 & -34466080 & 104755556 & -69672960 & 15801488 & 15801488 & -1575936 \\ 
-54645562 & 15801488 & 15801488 & -69672960 & 214917120 & 22387584 & 22387584 & 70060032 \\ 
-54645562 & 15801488 & -69672960 & 15801488 & 22387584 & 214917120 & 22387584 & 70060032 \\ 
-54645562 & -69672960 & 15801488 & 15801488 & 22387584 & 22387584 & 214917120 & 70060032 \\ 
-93533184 & -1575936 & -1575936 & -1575936 & 70060032 & 70060032 & 70060032 & 346816512\end{bmatrix}\]
Again, the verification of the hypotheses of Lemma~\ref{lem:flagLem} has been included in Appendix~B in an ancillary file submitted with the arxiv version of this paper:  \url{https://arxiv.org/src/2208.02045/anc/commonExtraAppendices.pdf}. 
\end{proof}

The ideas presented in this section and calculations in the appendices provide enough information to verify the proof of Theorem~\ref{th:Flag}, but say nothing about the way in which the proof was discovered. Clearly, to find the eight matrices that we used in the proof by hand would be a nightmare. Fortunately, the problem of finding such matrices can be phrased in terms of a semidefinite program (SDP) which can be solved by a computer. If one is trying to prove that a pair $(H_1,H_2)$ of graphs on at most $5$ vertices is $(p_1,p_2)$-common using the flags $F_{i,a}$ in this section, then the relevant semidefinite program asks for the maximum possible value of some variable, say $t$, subject to the constraints 
\[t\leq \frac{t_{\inj}(H_1,J_s)}{e(H_1)p_1^{e(H_1)-1}} + \frac{t_{\inj}(H_2,\overline{J_s})}{e(H_2)p_2^{e(H_2)-1}} - \sum_{i=1}^4\sum_{a=1}^8\sum_{b=1}^8M_i(a,b)c(F_{i,a},F_{i,b},J_s),\quad \forall s\in\{1,2,\dots,34\},\]
\[M_i\succcurlyeq 0, \quad\forall i\in\{1,2,3,4\}.\]
Note that the only variables in this SDP are $t$ and the entries of the matrices $M_1,M_2,M_3$ and $M_4$. A standard SDP solver can approximate the optimal value of such an SDP very quickly. The computer output, of course, is in the form of floating point approximation of the optimal value for $t$ and four $8\times 8$ matrices $M_1,M_2,M_3$ and $M_4$ consisting of floating point numbers. Thus, if the value of the output is close to  $\frac{p_1}{e(H_1)}+\frac{p_2}{e(H_2)}$, then it is strong evidence that $(H_1,H_2)$ is $(p_1,p_2)$-common, but not enough for a rigorous proof. To make the proof rigorous, one needs to ``round'' the floating point numbers in the matrices to exact values and verify that the rounded matrices give an optimal value of $t$ of exactly $\frac{p_1}{e(H_1)}+\frac{p_2}{e(H_2)}$. The task of rounding the entries in the matrices is non-trivial and tends to require a lot of trial and error; we will not go into details here. This describes the process that we followed to discover the proof of Theorem~\ref{th:Flag}. Many of the applications of flag algebras in the literature for minimizing expressions involving graph densities follow a similar recipe, but with a different choice of flags to suit their particular application.

\section{More Colours}
\label{sec:multicol}

Before closing the paper, let us propose a natural multicolour generalization of Definition~\ref{defn:commonPair}. We remark that the generalization of the notion of common graphs to more colours is studied in, e.g.,~\cite{JaggerStovicekThomason96,Kral+22,CummingsYoung11}.

\begin{defn}
\label{defn:commonMulticol}
Let $k\geq2$ and $p_1,p_2,\dots,p_k\in (0,1)$ such that $\sum_{i=1}^kp_i=1$ and let $H_1,H_2,\dots,H_k$ be non-empty graphs. We say that $(H_1,H_2,\dots,H_k)$ is \emph{$(p_1,p_2,\dots,p_k)$-common} if, for any graphons $W_1,W_2,\dots,W_k$ such that $\sum_{i=1}^kW_i=1$,
\[\sum_{i=1}^k\frac{t(H_i,W_i)}{e(H_i)p_i^{e(H_i)-1}}\geq \sum_{i=1}^k\frac{p_i}{e(H_i)}.\]
\end{defn}

Following~\cite{JaggerStovicekThomason96,Kral+22}, a graph $H$ is said to be \emph{$k$-common} if the $k$-tuple $(H,\dots,H)$ is $(1/k,\dots,1/k)$-common. As is the case with the ordinary notion of $k$-common graphs, it is more difficult to satisfy Definition~\ref{defn:commonMulticol} when $k$ is large. The following lemma is an asymmetric analogue of~\cite[Theorem~13]{JaggerStovicekThomason96}. 

\begin{lem}
\label{lem:kCol}
Let $k\geq2$ and $p_1,\dots,p_k,q_{k+1}\in (0,1)$ such that $\sum_{i=1}^kp_i=1$, define $q_i=(1-q_{k+1})p_i$ for $1\leq i\leq k$ and let $H_1,\dots,H_k, H_{k+1}$ be graphs. If $(H_1,\dots,H_{k})$ is not $(p_1,\dots,p_k)$-common, then $(H_1,\dots,H_{k+1})$ is not $(q_1,\dots,q_{k+1})$-common.
\end{lem}

\begin{proof}
Using the fact that $(H_1,\dots,H_{k})$ is not $(p_1,\dots,p_k)$-common, we let $W_1,\dots,W_k$ be graphons such that $\sum_{i=1}^kW_i=1$ and 
\[\sum_{i=1}^k\frac{t(H_i,W_i)}{e(H_i)p_i^{e(H_i)-1}}< \sum_{i=1}^k\frac{p_i}{e(H_i)}.\]
Define $W_i':=(1-q_{k+1})W_i$ for $1\leq i\leq k$ and $W_{k+1}'=q_k$. Then $\sum_{i=1}^{k+1}W_i'=1$ and
\[\sum_{i=1}^{k+1}\frac{t(H_i,W_i')}{e(H_i)q_i^{e(H_i)-1}}= \sum_{i=1}^k\frac{(1-q_{k+1})^{e(H_i)}t(H_i,W_i)}{e(H_i)(1-q_{k+1})^{e(H_i)-1}p_i^{e(H_i)-1}} + \frac{q_{k+1}^{e(H_{k+1})}}{e(H_{k+1})q_{k+1}^{e(H_{k+1})-1}}
<\sum_{i=1}^{k+1}\frac{q_i}{e(H_i)}.\]
This completes the proof. 
\end{proof}

A straightforward generalization of the proof of Theorem~\ref{th:Sid} shows that, if $H_1,\dots,H_k$ are Sidorenko, then $(H_1,\dots,H_k)$ is $(p_1,\dots,p_k)$ common for any $p_1,\dots,p_k\in (0,1)$ such that $\sum_{i=1}^kp_i=1$. The first non-bipartite examples of $k$-common graphs for $k\geq3$ were obtained recently in~\cite{Kral+22}; as shown in~\cite{KralVolecWei22+}, $k$-common graphs can have arbitrary chromatic number. It is also known that a graph $H$ is $k$-common for all $k\geq2$ if and only if it is Sidorenko~\cite[Theorem~2]{Kral+22}. Apart from these results, not much is known about $(p_1,p_2,\dots,p_k)$-common tuples of graphs; it would be interesting to explore this concept in more depth. 

We close this section with a necessary condition on $(p_1,p_2,p_3)$-common tuples of graphs in terms of their girth; note that, by Lemma~\ref{lem:kCol}, this lemma implies a necessary condition on $(p_1,\dots,p_k)$-common graphs for any $k\geq3$. This result essentially generalizes~\cite[Theorem~3]{Kral+22}, which is a strengthening of the main result of~\cite{CummingsYoung11}.

\begin{thm}
\label{th:GirthKCol}
Let $H_1,H_2$ and $H_3$ be graphs and let $p_1,p_2,p_3\in (0,1)$ such that $p_1+p_2+p_3=1$. If $\min\{g(H_1),g(H_2),g(H_3)\}$ is odd, then $(H_1,H_2,H_3)$ is not $(p_1,p_2,p_3)$-common. 
\end{thm}

\begin{proof}
Let $k=\min\{g(H_1),g(H_2),g(H_3)\}$. Without loss of generality, we can assume that 
\begin{equation}\label{eq:ckbound}\frac{c_k(H_1)}{e(H_1)p_1^{k-1}}\geq \frac{c_k(H_2)}{e(H_2)p_2^{k-1}}\geq \frac{c_k(H_3)}{e(H_3)p_3^{k-1}}.\end{equation}
For $i\in\{1,2,3\}$, let $b_i$ be the number of subgraphs of $H_i$ with at least $k+1$ non-isolated vertices. 

Now, let $B$ be the kernel as in the proof of Theorem~\ref{th:Girth}, let $p=\frac{1}{2}\min\{p_1,p_2,p_3\}$, let $\delta$ be a small real number to be specified later and define $U_1=2p\cdot B^\delta$ and $U_2=U_3=-p\cdot B^\delta$. Define $W_1=p_1+U_1$ and $W_2=W_3=p_2+U_2$. By Lemma~\ref{lem:expansion},
\begin{equation}\label{eq:expandKCols}
\sum_{i=1}^3\frac{t(H_i,W_i)}{e(H_i)p_i^{e(H_i)-1}} - \sum_{i=1}^3\frac{p_i}{e(H_i)}
= \sum_{i=1}^3\sum_{\substack{E\subseteq E(H_i) \\ E \ne \emptyset}}\frac{t(H_i[E],U_i)}{e(H_i)p_i^{|E|-1}}.
\end{equation}
Since $U_1,U_2$ and $U_3$ are all $0$-regular, we get that $t(F,U_i)=0$ for any acyclic graph $F$ and $i\in\{1,2,3\}$. In particular, $t(H_i[E],U_i)=0$ whenever $H_i[E]$ has at most $k-1$ non-isolated vertices. If $H_i[E]$ has exactly $k$ non-isolated vertices, then $t(H_i[E],U_i)$ is non-zero if and only if $H_i[E]$ is a cycle of length $k$. Also, as in the proof of Theorem~\ref{th:Girth}, we have $t(C_k,U_1)=-(2p)^k\delta^k$ and $t(C_k,U_2)=t(C_k,U_3)=p^k\delta^k$. If $F$ has at least $k+1$ non-isolated vertices, then $t(F,U_i)\leq \delta^{k+1}$ for all $i\in\{1,2,3\}$ by Lemma~\ref{lem:scaleDown}. Letting $b_i$ be the number of subgraphs of $H_i$ with at least $k+1$ non-isolated vertices, the expression on the right side of \eqref{eq:expandKCols} is therefore bounded above by
\[\sum_{i=1}^3 \frac{c_k(H_i)\cdot t(C_k,U_i)}{e(H_i)p_i^{k-1}}+ \sum_{i=1}^3\frac{b_i\delta^{k+1}}{e(H_i)p_i^{|E(H_i)|-1}}\]
\[= \left(\frac{-c_k(H_1)(2p)^k}{e(H_1)p_1^{k-1}} + \frac{c_k(H_2)p^k}{e(H_2)p_2^{k-1}} + \frac{c_k(H_3)p^k}{e(H_3)p_3^{k-1}}\right)\delta^k+ \sum_{i=1}^3\frac{b_i\delta^{k+1}}{e(H_i)p_i^{|E(H_i)|-1}}.\]
Now, by \eqref{eq:ckbound} and the fact that $2^k>2$, we see that, if $\delta$ is chosen sufficiently small with respect to $H_1,H_2,H_3,p_1,p_2$ and $p_3$, the above expression is negative. 
\end{proof}

\begin{rem}
\label{rem:localKCol}
Analogously to Remark~\ref{rem:local}, we remark that the proof of Theorem~\ref{th:GirthKCol} can be adapted to show that $(H_1,H_2,H_3)$ is not $(p_1,p_2,p_3)$-common in a local sense by simply scaling $U_1,U_2$ and $U_3$ by a factor of $\varepsilon$ for some arbitrarily small $\varepsilon$ and letting $\delta$ depend additionally on $\varepsilon$. After applying this modification, Theorem~\ref{th:GirthKCol} now generalizes~\cite[Theorem~3]{Kral+22}.
\end{rem}

\section{Conclusion}
\label{sec:concl}

We conclude the paper by stating several open problems. We start by repeating the four open problems stated in the introduction.

\repeatConj{Interval}

\repeatQues{Sid}

\repeatProb{CFourCFive}

\repeatConj{KFourBarrier}

Our last two open problems concern the strong barrier to commonality posed by $K_4$ subgraphs. It seems to us that $K_4$ may be the only subgraph which automatically prevents a graph from being common. 

\begin{conj}
\label{conj:K4exceptional}
For any $K_4$-free graph $F$, there exists a connected common graph $H$ containing $F$ as a subgraph. There is also a connected common graph $H$ containing $F$ as an induced subgraph. 
\end{conj}

Some strong evidence in favour of Conjecture~\ref{conj:K4exceptional} comes from the recent paper of Kr\'a\v{l}, Volec and Wei~\cite{KralVolecWei22+}, where it is shown that every graph of girth at least 50 is an induced subgraph of a connected common graph. 

An idea that is used in~\cite{Kral+22,KralVolecWei22+} is that the presence of large complete bipartite subgraphs in a graph $H$ tends to ``drive up'' the homomorphism density. Thus, a natural approach to Conjecture~\ref{conj:K4exceptional} is to simply consider a \emph{blow-up} of the graph $F$; that is, a graph obtained by replacing all vertices of $F$ by non-empty independent sets (of possibly different sizes) and adding all edges between two such sets if their corresponding vertices are adjacent. We conjecture that this strategy works to produce a common graph. 

\begin{conj}
\label{conj:K4BlowUp}
If $F$ is non-empty and $K_4$-free, then there exists a blow-up of $F$ which is common. 
\end{conj}

It is known that, for every bipartite graph $F$, there exists a blow-up of $F$ that is Sidorenko~\cite{ConlonLee17}; therefore, Conjecture~\ref{conj:K4BlowUp} holds for bipartite graphs (in a strong sense). 

There has been some significant recent interest in analogues of Sidorenko's Conjecture and the notion of common graphs in other areas, such as additive combinatorics~\cite{KamcevLiebenauMorrison21,SaadWolf17,FoxPhamZhao21,KamcevLiebenauMorrison21,Versteegen21,Versteegen23}. Here, the aim is to classify equations, or systems of equations, over $\mathbb{F}_q$ with the property that the number of solutions to the system in a subset $A\subseteq \mathbb{F}_q^n$ of a given size (in the Sidorenko problem) or a set $A$ and its complement (in the commonness problem) is asymptotically minimized when $A$ is chosen randomly. It may be interesting to study asymmetric versions of such problems. 

For posterity, we remark that the early arXiv pre-print version of this paper contained additional results and open problems. After submitting to arXiv, we decided to split the paper into two: the present paper and~\cite{SecondPaper}. 

\begin{ack}
The authors would like to thank Elena Moss for helping us make the pictures of flags in Section~\ref{sec:flags}. 
\end{ack}

\bibliographystyle{plain}

\end{document}